\documentclass{amsart}
\usepackage{amssymb}
\usepackage{amsmath}
\usepackage{color}
\usepackage[T1]{fontenc}

\usepackage{mathrsfs}
\usepackage{tikz}

\newcounter{thm}[section]

\newtheorem{lem}[thm]{Lemma}
\newtheorem{twr}[thm]{Theorem}

\newtheorem{cor}[thm]{Corollary}

\newtheorem{prop}[thm]{Proposition}

\newtheorem{fact}[thm]{Fact}

\theoremstyle{definition}
\newtheorem{remark}[thm]{Remark}

\newcommand{\frakG}{\mathcal Q}
\newcommand{\frakP}{\mathcal P}
\newcommand{\frakR}{\mathcal R}
\newcommand{\frakN}{\mathcal K}

\def\Fr{\operatorname{Fr}}

 \def\er{\mathbb{R}}
 \def\cc{\mathbb{C}}

\def\op{W_P}

\def\oop{\omega (P)}        

\def\opo{W_{\omega(P)}}

\def\la{\lambda } 
\def\La{\Lambda }
\def\re{\mathop{\rm Re}\nolimits}

\def\im{\mathop{\rm Im}\nolimits}

\def\const{\mathop{\rm const}\nolimits}

\def\en{\mathbb{N}}
\def\eq{\mathbb{Q}}

\def\K{\mathbb{K}}

\def\Int{\,{\rm Int}\,}

\def\ord{\mathop{\rm ord}\nolimits}

\def\ve{\varepsilon}
\def\WW{\operatorname{\mathcal W}}	
\title[A geometric model]{A geometric model\\ of an arbitrary differentially closed field\\ of characteristic zero}

\thanks{This research was partially supported by the Polish
National Science Centre, grant 2012/07/B/ST1/03293.}
\author[S. Spodzieja]{Stanis{\l}aw Spodzieja}
\address{Faculty of Mathematics and Computer Science, University of \L \'od\'z\newline
S. Banacha 22, 90-238 \L \'od\'z, Poland}
\email{stanislaw.spodzieja@wmii.uni.lodz.pl}
 
\begin{document}	
 \baselineskip=15pt
 
\maketitle


\begin{abstract} We give an elementary construction of an arbitrary differentially closed field {\color{black}and 
of a universal extension} of a differential field in terms of Nash function fields. 
 We also give a characterization of any 
Archime\-dean ordered differentially closed field in terms of Nash functions. 
\end{abstract}

\font\ccc=cmr10 
\font\ddd=cmti10

\section*{Introduction}

The study of differential algebras was started in the first half of the twentieth century by 
 J. F. Ritt \cite{Ritt0,Ritt}, and continued by E.~R.~Kolchin and J. F. Ritt \cite{RittKolchin} (see also \cite{Kolchin,Kolchin2}), I.~Kaplansky \cite{Kaplansky} and others  (see for instance \cite{Hoeven,Marker,PiercePillay,Robinson-2,Robinson-1,Robinson,Sacks,Seidenberg1, Seidenberg2}).  The investigation of these algebras in the context of model theory was initiated by A.~Robinson~\cite{Robinson-3}. Despite a fairly long period of study of differential algebras, it is difficult to indicate papers where natural examples of differentially closed fields are given. 
By A. Seidenberg's embedding theorem  \cite{Seidenberg1, Seidenberg2} we only know that: \emph{Every countable ordinary differential field of characteristic zero $F$ is differentially isomorphic over $F$ to a differential subfield of the field of germs of meromorphic functions in one variable at the origin}. {\color{black} L. Harrington \cite{Harrington} proved that  
\emph{if a complete and model complete decidable theory $T$ has the finite basis property and every quantifier-free constrained formula (in the language of  $T$) is complete, then $T$ has a recursively presentable prime model}. He used this model-theoretic result to construct the differential closure of any given recursively presentable differential field.}


The aims of this article are: to give {\color{black}models of ordinary differentially closed fields} of characteristic zero {\color{black}(Theorems \ref{charactdiffclosedfield} and \ref{maindifferclosed1})}; 
to construct {\color{black}a universal extension} of an ordinary differential field (Theorem \ref{coruniversal}); and to construct an Archimedean ordinary ordered differentially closed field (Theorem  \ref{maindifferclosed1real}). To this end we present a construction, in terms of Nash functions, of all algebraically closed fields of characteristic zero, i.e., the algebraic {\color{black}closure} of the rational function field $k=\eq(\La_T)$ in the system of independent variables $\La_T=(\La_t:t\in T)$, $T\ne\emptyset$, 
 with coefficients in $\eq$ (see \cite{Sppjm}). {\color{black} It suffices to consider such fields}, 
 because any ordinary differentially closed field $K$ of characteristic zero is differentially isomorphic to the algebraic closure
 of some field $\eq(\La_T)$ (Theorem \ref{charactdiffclosedfield}). If $T=\emptyset$ then $\eq(\La_T)=\eq$ and so the differential closure of  $\eq(\La_T)$ is contained in the algebraic closure of $\eq(\La_\en)$, i.e., one can take $T=\en$ (Proposition \ref{differentiallyclosedtranscendence}). 
We assume the Kuratowski-Zorn Lemma {\color{black}(and indirectly the axiom of choice, see \cite{KuratowskiMostowski})}, so the set $T$ can be well-ordered if $T\ne\emptyset$. 

The construction of any differentially closed field will be based on the construction of some family $\varOmega$ of open connected semialgebraic subsets of  $\cc^T$, called a $c$-filter (see Section~\ref{preliminaries}) and the rings $\mathcal{N}(U)$ of complex Nash functions on sets $U\in\varOmega$. 
 The algebraic closure of $\eq(\La_t:t\in T)$ will be constructed as the set of equivalence classes of the following relation in $\bigcup_{U\in \varOmega}\mathcal{N}(U)$:
$$
\hbox{$(f_1:U_1 \to \cc )\sim (f_2:U_2 \to \cc )$ iff $f_1|_{U_3}=f_2|_{U_3}$ for some $U_3\in\varOmega$.} 
$$
Then the set $\mathcal N_{\varOmega}$ of equivalence classes of  ``$\sim$'' with the usual operations of addition and multiplication is a field, which is the algebraic closure of $\eq(\La_T)$ (see Proposition \ref{realandalgebraicclos}, cf. \cite[Theorem 2.4 and Corollary 2.5]{Spodzieja2}). 
Whenever the space $\cc^T$ is infinite-dimensional, we will construct 
 a derivation $\delta$ on $\mathcal{N}_\varOmega$ such that for each pair $p,q \in \mathcal{N}_\varOmega\{y\}$ of differential polynomials such that 
 $\ord q< \ord p$, $q\ne0$, there is some $f\in \mathcal{N}_\varOmega$ with $p(f) = 0$ and $q(f)\ne 0$ (see Theorem \ref{maindifferclosed1}), {\color{black}i.e.,
%
   $\delta$ satisfies the L.~Blum \cite{Blum} definition of an ordinary differential closed field of characteristic zero.}

{\color{black} 
{\color{black}To build various kinds of differentially closed fields we construct two $c$-filters, $\varOmega^{\mathbb{K}}_{{\bf x_0}}$ and  $\WW^{\mathbb{K}}_T$, in $\K^T$, where $\K=\er$ or $\K=\cc$ (see Section \ref{cfilterssection} and Proposition \ref{realNash1} in Section \ref{ordering in Rm}).} 
All sets in those $c$-filters  will be 
 simply connected. {\color{black}Moreover, each $U\in\WW^\cc_T$} is dense in  $\cc^T$.
This  enables us to construct, for any ordinary differential Nash field, a differentially closed extension of that field of the same cardinality (see Corollary \ref{diffclosedextension}) and to construct {\color{black}a universal extension} of an ordinary differential field (Theorem \ref{coruniversal}). We also construct, in terms of real Nash functions on  $U\in \varOmega_{{\bf x_0}}^\er$, a model of an arbitrary ordinary Archimedean ordered differentially closed field (see Section \ref{constrorderedalgebraically}).}

\section{Differential fields}

In this section we will collect some fact concerning differential fields. For more detailed information on this topic see for instance \cite{Kaplansky, Kolchin,Marker,PiercePillay,Ritt,Robinson,
Sacks}.

\subsection{Differential algebras}

Let $k$ be a commutative ring with unity and let $A$ be a $k$-algebra, i.e., a left $k$-module with multiplication. A $k$-linear mapping $\delta:A\to A$ is called a \emph{derivation} on $A$ if 
$\delta (ab)=\delta (a)b+a\delta (b)$ for any $a,b\in A$. {\color{black} Then obviously all elements of $k$ are necessarily constants, i.e., $\delta(\lambda)=0$ for $\lambda\in k$.}

A $k$-\emph{differential algebra} $(A,\Delta)$ is defined as a $k$-algebra $A$ with a nonempty set $\Delta$ of derivations on $A$ such that $\delta\delta' (a)=\delta'\delta (a)$ for all $a\in A$ and $\delta,\delta'\in \Delta$. If the $k$-algebra $A$ is a ring, an integral domain or a field, then we call the $k$-differential algebra $(A,\Delta)$ a $k$-\emph{differential ring}, a $k$-\emph{differential domain} or a $k$-\emph{differential field} respectively. If $k=\eq$, we will write ``differential'' instead of ``$k$-differential''. If  $m=\operatorname{card} \Delta=1$, the $k$-differential algebra (respectively ring, domain or field) is called \emph{ordinary}; if $m>1$, it is called \emph{partial}. If $\Delta=\{\delta\}$, the $k$-differential algebra $(A,\Delta)$ is denoted by $(A,\delta)$.

Let $(A,\Delta)$ be a $k$-differential algebra and let $\Theta$ be the set of $k$-\emph{derivative operators} on $A$ generated by $\Delta$, i.e., the 
free commutative semigroup generated by $\Delta$. Then any $\theta\in\Theta$ can be uniquely expressed in the form of a product $\theta=\prod_{\delta\in\Delta}\delta^{e(\delta)}$, where $e(\delta)\in\en$ (we assume that $0\in\en$). The number $s=\sum_{\delta\in\Delta}e(\delta)$ is called the \emph{order} of $\theta$ and is denoted by $\ord \theta$.   

\subsection{Differential polynomials}

Let $(R,\Delta)$ be a $k$-differential ring, $\Theta$  the set of $k$-derivative operators on $R$ generated by $\Delta$, and $J$ a nonempty set.  We denote by $R\{y_j:j\in J\}$ the \emph{ring of $k$-differential polynomials}, i.e., the ring of polynomials 
$$
R\{y_j:j\in J\}:=R[Y_{J,\Theta}] 
$$
with coefficients in $R$, in the system of variables $Y_{J,\Theta}=(y_{j,\theta}:j\in J,\; \theta\in\Theta)$, where we assume that $\theta_1 (y_{j,\theta_2})=y_{j,\theta_1\theta_2}$ for $\theta_1,\theta_2\in\Theta$ and $y_j=y_{j,\theta}$ for $\theta\in \Theta$ of order~$0$. The ring $R\{y_j:j\in J\}$ has the structure of a $k$-differential ring with the set of $k$-derivations $\Delta$  if we set $\delta (y_{j,\theta})=y_{j,\delta\theta}$ for $j\in J$, $\delta\in\Delta$ and $\theta\in\Theta$. If $\operatorname{card}J=1$, we will write $R\{y\}$ instead of $R\{y_j:j\in J\}$. 

Take any $k$-differential polynomial $p\in R\{y_j:j\in J\}$. Then there are $d,n\in\en$, $n>0$, and $(j_1,\theta_1),\ldots,(j_n,\theta_n)\in J\times \Theta$ such that 
\begin{equation*}\label{eqfprmp}
p(Y_{J,\Theta})=\sum_{i_1+\cdots+i_n\le d}r_{i_1,\ldots,i_n}(y_{j_1,\theta_1})^{i_1}\cdots (y_{j_n,\theta_n})^{i_n},
\end{equation*}
where $i_1,\ldots,i_n\in\en$ and $r_{i_1,\ldots,i_n}\in R$ for $i_1+\cdots+i_n\le d$. The number 
$$
\max\{i_1+\cdots+i_n:r_{i_1,\ldots,i_n}\ne 0,\,i_1+\cdots+i_n\le d\}
$$ 
is called the \emph{degree} of $p$ and denoted by $\deg p$ ($\deg p=-\infty$ if $p=0$). The number 
$$
\max\bigcup_{1\le s\le n}\{\ord \theta_s: r_{i_1,\ldots ,i_n}\ne 0,\, i_1+\cdots+i_n\le d,\, i_s>0\}
$$ 
is called the \emph{order} of $p$ and denoted by $\ord p$ (we set $\max\emptyset=-1$, and then $\ord p=-1$ if $p\in R$).

If $\operatorname{card}J=1$ and $\Delta=\{\delta\}$, then for any polynomial $p\in R\{y\}$, we denote by $p^*$  the unique polynomial from $R[x_0,\ldots,x_n]$, where $n={\ord p}$, such that 
$$
p(y)=p^*(y_{\delta^0},\ldots,y_{\delta^n}).
$$
Then for any $a\in R$ we have
$$
p(a)=p^*(a,\delta(a),\ldots,\delta^n(a)).
$$

\subsection{Differentially closed fields}

A field $K$ of characteristic zero equipped with $m>0$ commuting derivations
is called \emph{differentially closed} (or \emph{partial differentially closed} for $m>1$) if  every system of differential polynomial equations {\color{black}and
inequations} in several variables with a solution in some differential extension
 of $K$  has a solution in $K$. If  $m=1$, then the ordinary differential field $K$ is called \emph{ordinary differentially closed}. 

In this paper we will use the following (equivalent) definition of ordinary differentially closed fields, due to L. Blum~\cite{Blum}:

An ordinary differential field $(K,\delta)$ of characteristic zero  is called 
\emph{differentially closed}
if for each pair $p, q \in K\{y\}$ of differential polynomials such that 
$\ord q< \ord p$ and $q\ne 0$, there is some $a\in  K$ with $p(a) = 0$ and $q(a)\ne 0$.

From \cite[Lemma 4]{Shelah} we obtain the following fact (cf. \cite{Seidenberg0,Seidenberg,Seidenberg1,Seidenberg2}). 

\begin{prop}\label{differentiallyclosedtranscendence}
Assume that $(K,\delta)$ is a differentially closed field of characteristic zero. Then the transcendence degree $\operatorname{trdeg}_\eq K$ of $K$ over $\eq$ is infinite. 
\end{prop}

\begin{proof} 
We will write $y'$ for $y_\delta$ and $y$ for $y_{0}$. 
Consider the  differential polynomials $p_0=(1+y)y'-y$, $q_0=y$.  Then $0\le \ord q_0<\ord p_0$. So, there exists  $\varphi_0\in K$ such that $p_0(\varphi_0)=0$ and $q_0(\varphi_0)\ne 0$. Consider a sequence of differential polynomials $p_j=p_0$ and $q_j=q_{j-1}(y-\varphi_{j-1})$, where $\varphi_{j-1}\in K$, $p_{j-1}(\varphi_{j-1})=0$ and $q_{j-1}(\varphi_{j-1})\ne 0$ for $j\in\en$, $j>0$. Since $(K,\delta)$ is a differentially closed field, such sequences  exist. Consequently, there exist an infinite number of distinct nonzero solutions of the equation $p_0=0$. Then  \cite[Lemma 4]{ Shelah} yields the assertion.
\end{proof}

{\color{black}Let the differential field $(\mathcal{U},\delta^*)$ be a differential extension of a differential field $(K,\delta)$ of characteristic zero, i.e., $K$ is a subfield of $\mathcal{U}$ and $\delta^*(a)=\delta(a)$ for $a\in K$. We say that the extension $(\mathcal{U},\delta^*)$ of $(K,\delta)$ is \emph{finitely generated} if $\mathcal{U}$ has a finite subset $A$ such that $(\mathcal{U},\delta^*)$ is the smallest differential extension of $(K,\delta)$ in $(\mathcal{U},\delta^*)$ that contains $A$. The set $A$ is called the \emph{set of generators} of the extension $(\mathcal{U},\delta^*)$ over $(K,\delta)$. The extension $(\mathcal{U},\delta^*)$ over $(K,\delta)$ is called \emph{simply generated} if
  it has a  set of generators consisting of one element.

After E.R. Kolchin \cite{Kolchin,Kolchin2}, we say that $(\mathcal{U},\delta^*)$ is a \emph{semiuniversal extension} of  $(K,\delta)$ 
if every finitely generated differential extension of $(K,\delta)$ differentially embeds over $K$ in $(\mathcal{U},\delta^*)$. We say that $(\mathcal{U},\delta^*)$ is a \emph{universal extension} of  $(K,\delta)$ if $(\mathcal{U},\delta^*)$ is semiuniversal over every finitely generated differential extension of $(K,\delta)$. A~universal extension of the field $\eq$ is called a \emph{universal differential field}.}

\subsection{Ordered differentially closed fields}

We will use the following  definition of ordered  ordinary differentially closed fields, due to M. Singer \cite{Singer}.

Let $R$ be a real field with an ordering $\succ$ and a derivation $\delta$. The field $(R,\delta)$ is called an \emph{ordered  \emph{(or }real\emph{)} ordinary differentially closed field}, or briefly \emph{ordered differentially closed}, if $R$ is real closed and 
for any $p,q_1,\ldots,q_m\in R\{y\}$ such that $n=\ord p\ge \ord q_i$ for $1\le i\le m$, and any $a_0,\ldots,a_n\in R$ such that $p^*(a_0,\ldots,a_n)=0$ with $\frac{\partial p^*}{\partial x_n}(a_0,\ldots,a_n)\ne 0$ and $q_i^*(a_0,\ldots,a_n)\succ 0$ for $1\le i\le m$, there exists $a\in R$ such that $p(a)=0$ and $q_i(a)\succ 0$ for $0\le i\le m$.

\begin{remark}\label{densitythm1}
From the definition we immediately see that if $(R,\delta)$ is an ordered differentially closed field, 
then for any $n\in\en$ the space
$$
V=\{(a,\delta(a),\ldots,\delta^n(a))\in R^{n+1}:a\in R\}
$$
is dense in $R^{n+1}$, i.e., for any polynomials $g_1,\ldots,g_m\in R[x_0,\ldots, x_n]$, if 
$$
X=\{(a_0,\ldots,a_n)\in R^{n+1}:g_i(a_0,\ldots,a_n)\succ 0,\;1\le i\le m\}\ne \emptyset,
$$
then $V\cap X\ne \emptyset$.
\end{remark}

The following is known (see \cite{Singer2}):

\begin{prop}\label{closedordereddifferential}
If $(R,\delta)$ is an ordered differentially closed field, then $(R(i),\delta^*)$, where $i^2=-1$ and  $\delta^*(f_1+if_2)=\delta(f_1)+i\delta(f_2)$ for $f_1,f_2\in R$, is a differentially closed field $($of characteristic $0)$. 
\end{prop}

From the above and Proposition \ref{differentiallyclosedtranscendence} we have

\begin{cor}\label{orderedtranscendence}
Assume that $(R,\delta)$ is an ordered differentially closed field. Then the transcendence degree $\operatorname{trdeg}_\eq R$ of $R$ over $\eq$ is infinite. 
\end{cor}

{\color{black}Let $R$ be a real closed field ordered by $\succ$. For $x=(x_1,\ldots,x_n)\in R^n$ we denote $\|x\|=\sqrt{x_1^2+\cdots+x_n^2}$. The  \emph{Euclidean topology} in $R^n$ is the topology for which the open balls $B(x,r):=\{y\in R^n:\|x-y\|\prec r\}$, $x\in R^n$, $r\in R$, $r\succ 0$, form a
basis of open subsets (see \cite[Definition 2.1.9]{BochnakCosteRoy}). Polynomials are continuous with respect to the Euclidean topology.}

\begin{prop}\label{Hoeven}
Let $(R,\delta)$ be an ordered differentially closed field, ordered by $\succ$.\linebreak Then for any 
 $p\in R\{y\}$ and any $a,b\in R$ such that $a\prec b$ and $p(a)p(b)\prec 0$ there exists $c\in R$ such that $a\prec c\prec b$ and $p(c)=0$. {\color{black}Moreover, if $n=\ord p$ then $c$ can be chosen in such a way that ${\bf c}\in B({\bf s},\|{\bf s}-{\bf a}\|)$, where ${\bf c}=(c,\delta(c),\ldots,\delta^n(c))$, ${\bf s}=\frac{1}{2}({\bf a}+{\bf b})$ and ${\bf a}=(a,\delta(a),\ldots,\delta^n(a))$, ${\bf b}= (b,\delta(b),\ldots,\delta^n(b))$.} 
\end{prop}

{\color{black}
\begin{proof}
Take any $p\in R\{y\}$, $n=\ord p$, and $a,b\in R$ as in the assumption. Let $p^*\in R[x_0,\ldots,x_n]$ be the unique polynomial such that  $p(y)=p^*(y_{\delta^0},y_{\delta},\ldots,y_{\delta^n})$, and let ${\bf a}=(a,\delta(a),\ldots,\delta^n(a))$, ${\bf b}= (b,\delta(b),\ldots,\delta^n(b))$, ${\bf s}=\frac{1}{2}({\bf a}+{\bf b})$.

One can assume that the polynomial $p^*$ is irreducible in $R[x_0,\ldots,x_n]$ and its degree with respect to $x_n$ is positive.  Since $p^*({\bf a})p^*({\bf b})\prec 0$, the sign of the polynomial $p^*$ changes in $R^{n+1}$ and by \cite[Theorem 4.5.1]{BochnakCosteRoy}, the ideal $(p^*)\subset R[x_0,\ldots,x_n]$ is real, prime and it is the ideal of polynomials 
 vanishing on the hypersurface $V=\{x\in R^{n+1}:p^*(x)=0\}$. Moreover, $\dim V=n$. Consequently, $\frac{\partial p^*}{\partial x_n}\notin (p^*)$ and $\dim W<n$, where  $W=\{x\in V :\frac{\partial p^*}{\partial x_n}(x)=0\}$.

Take polynomials $q_1,q_2\in R\{y\}$ defined by $q_1(y)=\|{\bf s}-{\bf a}\|^2-\|(y_{\delta^0},\ldots,y_{\delta^n})-{\bf s}\|^2$, $q_2(y)=(y-a)(b-y)$. 
Then $\ord q_1=n$, $\ord q_2=0$ and $q_1(y)=q_1^*(y_{\delta^0},y_{\delta},\ldots,y_{\delta^n})$ for $q_1^*(x_0,\ldots,x_n)=\|{\bf s}-{\bf a}\|^2-\|(x_0,\ldots,x_n)-{\bf s}\|^2$, and $q_2(y)=q_2^*(y_{\delta^0},y_{\delta},\ldots,y_{\delta^n})$ for $q_2^*(x_{0},\ldots,x_n)=(x_0-a)(b-x_0)$,  and $B({\bf s},\|{\bf s}-{\bf a}\|)=\{x\in R^{n+1}:q_1^*(x)\succ 0\}$. 

Take a polynomial $r\in R[t]$  defined by $r(t)=p^*(t{\bf a}+(1-t){\bf b})$. 
By the assumptions,  $r(0)r(1)\prec 0$. Since $R$ is real closed, there exists $t_0\in R$ with $0\prec t_0\prec 1$ such that $r(t_0)=0$ (see  \cite[Proposition 1.2.4]{BochnakCosteRoy}). We may assume that $r$ changes sign at the point $t_0$ (i.e., $r(t_0)=0$  and $r$ takes positive and negative values in any neighbourhood of $t_0$). Put ${\bf a}_0=t_0{\bf a}+(1-t_0){\bf b}$. Then ${\bf a}_0\in B({\bf s},\|{\bf s}-{\bf a}\|)$, $p^*({\bf a}_0)=0$ and $q_2^*({\bf a}_0)\succ 0$. So, there exists $\ve\succ 0$ such that  $B({\bf a}_0,\ve)\subset B({\bf s},\|{\bf s}-{\bf a}\|)$  and $q_2^*(x)\succ 0$ for $x\in B({\bf a}_0,\ve)$. Set $B=B({\bf a}_0,\ve)$. 

Let $U_1=\{x\in R^{n+1}:p^*(x)\succ 0\}$ and $U_2=\{x\in R^{n+1}:p^*(x)\prec 0\}$. Since $p^*$ is irreducible and changes sign in $B$, 
 we have $U_1\cap B\ne\emptyset$ and $U_2\cap B\ne \emptyset$. So, by \cite[Lemma 3.4.2]{BochnakCosteRoy}, $\dim (V\cap B)=\dim(B\setminus (U_1\cup U_2))=n$. Since $\dim W<n$, there exists ${\bf a}_1\in (V\cap B)\setminus W$, and so $p^*({\bf a}_1)=0$, $\frac{\partial p^*}{\partial x_n}({\bf a}_1)\ne 0$ and $q_j^*({\bf a}_1) \succ 0$, and obviously $\ord q_j \le \ord p$ for $j=1,2$. Then by definition of ordered differentially closed field, there exists $c\in R$ such that $p(c)=0$ and $q_j(c)\succ 0$, $j=1,2$. Consequently, ${\bf c}\in B$ and $a\prec c\prec b$, which completes the proof.\end{proof}

J. van der Hoeven  \cite{Hoeven} proved that the field $\mathbb{T}$ of transseries satisfies the first part ot the assertion of Proposition \ref{Hoeven}. 
It is not clear whether the converse of the van der Hoeven result holds for real differential fields. Nevertheless, we have the converse of the ``moreover'' part of  Proposition \ref{Hoeven}.

\begin{cor}\label{corequivrealdiffclosed}
Let $(R,\delta)$ be an ordered differential field, ordered by $\succ$. Assume that  $R$ is real closed. Then the following conditions are equivalent:

{\rm (a)} $(R,\delta)$ is an ordered differentially closed field.

{\rm (b)} For any $p\in R\{y\}$, $n=\ord p\ge 0$, and any ${\bf a},{\bf b}\in R^{n+1}$ with  $p^*({\bf a})p^*({\bf b})\prec 0$, there exists $c\in R$ such that $p(c)=0$ and ${\bf c}\in B({\bf s}, \|{\bf s}-{\bf a}\|)$, where ${\bf c}=(c,\delta(c),\ldots,\delta^n(c))$ and ${\bf s}=\frac{1}{2}({\bf a}+{\bf b})$. 

{\rm (c)} For any $p\in R\{y\}$, $n=\ord p\ge 0$, any  ${\bf a}\in R^{n+1}$ at which $p^*$ changes sign, and any ball $B({\bf a},r)$, $r\succ 0$, there exists $c\in R$ such that $p(c)=0$ and ${\bf c}\in B({\bf a},r)$, where ${\bf c}=(c,\delta(c),\ldots,\delta^n(c))$.

{\rm (d)} For any $p\in R\{y\}$, $n=\ord p\ge 0$, any ${\bf a}\in R^{n+1}$ with  $p^*({\bf a})=0$, $\frac{\partial p^*}{\partial x_n}({\bf a})\ne 0$, and any ball $B({\bf a},r)$, $r\succ 0$, there exists $c\in R$ such that $p(c)=0$ and ${\bf c}\in B({\bf a},r)$, where ${\bf c}=(c,\delta(c),\ldots,\delta^n(c))$.
\end{cor}

\begin{proof}  
The implication (a) $\Rightarrow$ (b) is proved similarly to Proposition \ref{Hoeven}. The implication (c) $\Rightarrow$ (d) is obvious.

If $p\in R\{y\}$ with $n=\ord p\ge 0$ changes sign at ${\bf a}\in R^{n+1}$, 
then for any ball $B({\bf a},r)$ with $r\succ 0$ there are ${\bf a}_1,{\bf b}_1\in B({\bf a},r)$ such that $p^*({\bf a}_1)p^*({\bf b}_1)\prec 0$ and $B({\bf s},\|{\bf s}-{\bf a}_1\|)\subset B({\bf a},r)$, where ${\bf s}=\frac{1}{2}({\bf a}_1+{\bf b}_1)$. Thus, (b) gives (c).

Take any $p,q_1,\ldots,q_m\in R\{y\}$ with $n=\ord p\ge \ord q_j$, $1\le j\le m$, and $p^*({\bf a})=0$, $\frac{\partial p^*}{\partial x_n}({\bf a})\ne 0$, and $q_j({\bf a})\succ 0$, $1\le j\le m$, for some ${\bf a}\in R^{n+1}$. 
Since polynomials are continuous in the Euclidean topology in $R^{n+1}$, there exists a ball $B({\bf a},r)$, $r\succ 0$, such that $q_i(x)\succ 0$ for $x\in B({\bf a},r)$, $1\le j \le m$. Thus (d) gives (a).
\end{proof}}

\section{Semialgebraic preliminaries}\label{preliminaries}

\subsection{$\eq$-algebraic and semialgebraic sets} 
Let us recall some facts from \cite{Sppjm} concerning algebraic and semialgebraic sets in a space of infinite dimensions.

{\color{black}Let $\K=\er$ or $\K=\cc$.}  
Let $T$ be a nonempty set. We denote by $\La_T=(\La_ t: t\in T)$  a system of independent variables, and by $\K[\La_T]$ and $\K(\La_T)$  the ring of polynomials in the variables of $\La_T$ over $\K$ and its quotient field, respectively. More precisely, for any $P\in\K(\La_T)$ we have $P\in \K(\La_{ t_1},\ldots,\La_{ t_m})$ for some finitely many $ t_1,\ldots,  t_m\in T$. 

 We denote by $\K^T$ the set of all functions $T\to \K$ equipped with the {\color{black}product} topology. Then 
  all projections $\K^T\ni x\mapsto x( t)\in\K$, $ t\in T$, are continuous. 

A subset of $\K^T$ is called \emph{$\eq$-algebraic} 
when it is defined by a finite system of equations $P=0$, where $P\in \eq[\La_T]$. Any $\eq$-algebraic set in $\K^T$ is of the form $\{x\in\K^T:(x( t_1),\ldots,x( t_m))\in V\}$, where $m\in\en$, $ t_1,\ldots, t_m\in T$ and $V\subset \K^m$ is  a $\eq$-algebraic subset of $\K^m$ {\color{black}(a complex $\eq$-algebraic set if $\K=\cc$)}. 

{\color{black}Note that any $\eq$-algebraic subset of $\cc^T$ is also $\eq$-algebraic in $(\er^T)^2$. However, a $\eq$-algebraic set in $(\er^T)^2$ is also $\eq$-algebraic in $\cc^T$ only if it is $\eq$-algebraic as a complex algebraic set.} 

A subset of $\K^T$ (if $\K=\cc$ we identify $\cc^T$ with $(\er^T)^2$) is called {\it $\eq$-semialgebraic} when it is defined by a {\color{black}finite boolean combination} 
of inequalities $P>0$ or $P\ge  0$, where $P\in\eq[\La_T]$ ($P\in\eq[\La_T,\La'_T]$, where {\color{black} $\La'_T=(\La'_t:t\in T)$ is a system of independent variables and $\La_T,\La'_T$ represent real and imaginary parts of complex numbers}, 
if $\K= \cc$). Analogously to the above, any $\eq$-semialgebraic set in $\K^T$ is of the form $\{x\in\K^T:(x( t_1),\ldots,x( t_m))\in X\}$, where $m\in\en$, $ t_1,\ldots, t_m\in T$ and $X\subset \K^m$ is a $\eq$-semialgebraic subset of $\K^m$.

From the basic properties of algebraic and semialgebraic sets in finite-dimensional real vector spaces (see \cite{BenedettiRisler}, \cite{BochnakCosteRoy}, \cite{BochnakEfroymson}, \cite{PrestelDelzell}) we obtain

\begin{prop}\label{semi1}

{\rm (a)} The family of $\eq$-algebraic sets in $\K^T$ is closed with respect to finite unions and intersections. 

{\rm (b)} The family of $\eq$-semialgebraic sets in $\K^T$ is closed with respect to complements and finite unions and intersections. 

{\rm (c)} {\rm({Tarski-Seidenberg}).} Let $\pi_{ t_1,\ldots, t_m} :\er^T\ni x\mapsto (x( t_1),\ldots,x( t_m))\in\er^m$, where $ t_1,\ldots,  t_m\in T$. If $X\subset \er^T$ and $Y\subset \er^m$ are $\eq$-semialgebraic, then so are $\pi_{ t_1,\ldots, t_m} (X)$ and $\pi_{ t_1,\ldots, t_m}^{-1}(Y)$.

{\rm (d)} For any $\eq$-semialgebraic set $X\subset \er^T$, the interior $\Int X$, closure $\overline{X}$ and the boundary $\Fr X$ are $\eq$-semialgebraic. 

{\rm (e)} Every connected component of a $\eq$-semialgebraic subset of $\er^T$ is $\eq$-semi\-alge\-braic.
\end{prop}

\subsection{$c$-filters}\label{cfilterssection}
Let $\K=\er$ or $\K=\cc$. A family $\varOmega$ of subsets of $\K^T$ satisfying the following conditions:

\noindent{\rm(i)}  any $U\in{\varOmega}$ is a nonempty open connected $\eq$-semialgebraic set,

\noindent{\rm(ii)} for any $\eq$-algebraic set $V\varsubsetneq \K^T$ there exists $U\in{\varOmega}$ such that $V\cap U=\emptyset$,

\noindent{\rm(iii)} for any $U_1,U_2\in {\varOmega}$ there exists $U_3\in \varOmega$ such that $U_3\subset U_1\cap U_2$,

{\color{black}\noindent{\rm(iv)} 
for any $U\in\varOmega$ there exist an open connected and simply connected\linebreak\indent\indent  $\eq$-semialgebraic set $U_0\subset \K^T$ and a set $U'\in\varOmega$ such that $U'\subset U_0\subset U$,}

\noindent will be  called a  \emph{$c$-filter} in $\K^T$ {\color{black}(cf. \cite{Sppjm})}.

{\color{black} Condition (iv) in the definition of $c$-filter is necessary in the construction of the $\cc$-field of Nash functions (see Section \ref{fieldofNashfunct}), because
the construction  is based on the monodromy theorem (see \cite[Theorem 2.4]{Spodzieja2} and \cite[Remark 5.6]{Sppjm}). In \cite{Sppjm} we used $c$-filters only in the real space $\er^T$ and we omitted  condition (iv), because it was unnecessary. In the real case a $\eq$-Nash function $ f: U \to \er $ such that $ P (\la, f (\la)) = 0 $, where $P\in \eq[\La_T,Z]$, is defined by $f(\la)=\xi_i(\la)$, $\la\in U$, for fixed $1\le i\le m$, where $ \xi_1(\la) <\cdots <\xi_m(\la) $ are roots of the polynomial $ P (\la, Z) $, provided the resultant of $P$ with respect to $Z$ has no zeros in $U$. However, in the real case we can use the results from \cite{Sppjm}, because then  condition (iv) follows from the others. Namely we have

\begin{prop}\label{simplyconnected}
Let $\varOmega$ be a family of subsets of $\er^T$ satisfying  conditions {\rm (i), (ii) and (iii)} in the definition of $c$-filter. Then $\varOmega$ also satisfies  condition {\rm(iv)}.
\end{prop}}
\begin{proof}
Take any $U\in\varOmega$. Then there exist $m\in\en$ and $t_1,\ldots,t_m\in T$  such that $U=\{x\in\er^T:(x( t_1),\ldots,x( t_m))\in X\}$ for some open connected $\eq$-semialgebraic set $X\subset \er^m$. Take a 
cylindrical decomposition $S_1,\ldots,S_\nu$ of $\er^m$ into $\eq$-semialgebraic sets adapted to the set $X$ {\color{black}(see \cite[Theorem 5.6 and Algorithm 11.15]{Basu})}. 
 We may assume that $X=\bigcup _{j=1}^n S_j$ for some $n\le \nu$, and  $\Int S_1,\ldots,\Int S_\ell\ne \emptyset$, while $\Int S_{\ell+1}=\cdots=\Int S_n=\emptyset$. Then $\overline{X}=\bigcup_{j=1}^\ell\overline{S_j}$, and $S_1,\ldots, S_\ell$ are open connected and simply connected $\eq$-semialgebraic sets. Moreover, $\bigcup _{j=1}^\ell\Fr S_j$  
is contained in some proper $\eq$-algebraic subset $W$ of~$\er^m$, {\color{black}where $\Fr S_j$ denotes the boundary of $S_j$}. Set 
\begin{equation*}
\begin{split}
U_j&=\{x\in\er^T:(x( t_1),\ldots,x( t_m))\in S_j\},\quad j=1,\ldots, \ell,\\
V&=\{x\in\er^T:(x( t_1),\ldots,x( t_m))\in W\}.
\end{split}
\end{equation*}
Then $U\setminus V\subset \sum_{j=1}^\ell U_j\subset U$, and by conditions (ii) and (iii) in the definition of $c$-filter,  there exists $U'\in\varOmega $ such that $U'\subset U\setminus V$, and by (i), $U'\subset U_j$ for some $j\in\{1,\ldots,\ell\}$. Then taking $U_0=U_j$ we deduce the assertion.
 \end{proof}

In the real case we have the following property of $c$-filters.
      
\begin{prop}[{\cite[ Proposition 2.1]{Sppjm}}]\label{Corollary1}
For any $c$-filter $\varOmega $ of subsets of $\er^T$,  
 the set $\partial \varOmega:= \bigcap_{U\in \varOmega}\overline{U}$ has at most one point. 
\end{prop}

The assertion of Proposition \ref{Corollary1} fails in the complex case (see {\color{black}Remark \ref{rempartialWTC}} in Section \ref{ordering in Rm}). However, it does hold for some $c$-filters in $\cc^T$. Namely, let 
 $T\subset \er$ be a set algebraically independent over $\eq$, and let ${\bf x}_0\in \er^T$ be  defined by ${\bf x}_0(t)=t$ for $t\in T$. Then there exists a $c$-filter $\varOmega^\K_{{\bf x}_0}$ of subsets of $\K^T$ 
  of the form
\begin{equation}\label{defarchordcenbtered}
U=\{x\in\K^T:|x(t_j)-\tilde x_j|<\varepsilon,\;j=1,\ldots, m\},
\end{equation}
for any $t_1,\ldots,t_m\in T$ with $t_1<\cdots<t_m$ and $\varepsilon\in \eq_{+}$ and $\tilde x=(\tilde x_1,\ldots,\tilde x_m)\in \eq^m$ such that $|t_j-\tilde x_j|<\varepsilon$ for  $j=1,\ldots,m$ (and so ${\bf x}_0\in U$).  

A $c$-filter $\varOmega$ in $\K^T$ such that ${\bf x}_0\in U$ for any $U\in \varOmega$ will be called \emph{centered} at $\bf x_0$.

\subsection{Field of Nash functions}\label{fieldofNashfunct} 

A function $f:U\to\er$, where $U\subset \er^T$ is an open $\eq$-semialgebraic set, is called a $\eq$-{\it Nash function} if $f$ is {\color{black} real}  analytic and  there exists a nonzero polynomial $P\in \eq[\La_T,Z]$ such that $P(\la,f(\la))=0$ for $\la\in U$.  In fact $f$ depends on a finite number of variables, so the analyticity of $f$ is clear. The ring of $\eq$-Nash functions in $U$ is denoted by 
${\mathcal N}^{\er}_{}(U)$.

A function $f:U\to\cc$, where $U\subset \cc^T$ is an open $\eq$-semialgebraic set (as a subset of $\er^T\times \er^T$), is called a $\cc$-$\eq$-{\it Nash function} if $f$ is holomorphic and there exists a nonzero polynomial $P\in \eq[\La_T,Z]$ such that $P(\la,f(\la))=0$ for $\la\in U$. The ring of $\cc$-$\eq$-Nash functions in $U$ is denoted by ${\mathcal N}^{\cc}_{}(U)$.  

{\color{black}Any nonzero polynomial $P\in \eq[\La_T,Z]$ determines at most $\deg_Z P$ $\eq$-Nash functions in a nonempty open connected $\eq$-semialgebraic set $U\subset \er^T$ (respectively $\cc$-$\eq$-Nash functions  in a nonempty open connected $\eq$-semialgebraic set $U\subset \cc^T$).}

For the basic properties of Nash functions and semialgebraic sets in finite-dimensional vector spaces see for instance \cite{BenedettiRisler}, \cite{BochnakCosteRoy}, \cite{BochnakEfroymson}, \cite{Nash}.  From these properties we immediately obtain:

\begin{prop}\label{Nash1} Let $\K=\er$ or $\K=\cc$, and let $U\subset \K^T$ be a nonempty open connected $\eq$-semialgebraic set. Then the ring ${\mathcal N}^{\K}_{}(U)$ is a domain. 
\end{prop}

By Proposition \ref{Nash1}, for any $c$-filter $\varOmega$ 
in $\K^T$ and any $U\in\varOmega$, the ring ${\mathcal N}^\K_{}(U)$ of $\eq$-Nash functions if $\K=\er$ or $\cc$-$\eq$-Nash functions if $\K=\cc$ on $U$ is a domain. In $\bigcup_{U\in \varOmega}{\mathcal N}_{}^\K(U)$ we introduce an equivalence relation by
$$
(f_1:U_1 \to \K )\sim (f_2:U_2 \to \K )\quad\hbox{iff}\quad f_1|_{U_3}=f_2|_{U_3}\hbox{ for some $U_3\in\varOmega$}.
$$
The $\sim$-equivalence class of  $f:U\to  \er $ will be denoted by $[f]_{\varOmega}$ or simply by $f$, and the set of all such classes  by ${\mathcal N}^\K_{\varOmega}$. The set ${\mathcal N}^\K_{\varOmega}$, together with the usual operations of addition and multiplication
$$
[f_1]_\varOmega+[f_2]_\varOmega=\bigl[ f_1|_{U}+f_2|_{U}\bigr]_\varOmega\,,\quad [f_1]_\varOmega\cdot [f_2]_\varOmega=\bigl[ f_1|_{U}f_2|_{U}\bigr]_\varOmega\,,
$$ 
where $f_1\in {\mathcal N}^\K_{}(U_1)$, $f_2\in {\mathcal N}^\K_{}(U_2)$ and $U\in\varOmega $, $U\subset U_1\cap U_2$, is a field, called the $\K$-\emph{field of Nash functions}.

{\color{black}From} \cite[Theorems 5.2 and Remark 5.6]{Sppjm} 
we obtain the following proposition. 

\begin{prop}\label{realandalgebraicclos}
Let $\varOmega$ be a $c$-filter 
 in $\K^T$.

{\rm (a)} If $\K=\er$, then the field $\mathcal N^\K_{\varOmega}$ is a real closure of the field $\eq(\La_T)$, where the  $c$-filter $\varOmega$ determines a  linear ordering $\succ_\varOmega$  in $\mathcal N^\K_{\varOmega}$  by (see Section \ref{orderingsineqlaDeltaa})
$$
f\succ_\varOmega g\quad\hbox{iff there exists }U\in \varOmega\hbox{ such that }f(x)>g(x)\hbox{ for all } x\in U.
$$

{\rm (b)} If $\K=\cc$, then the field $\mathcal N^\K_{\varOmega}$ is the algebraic  closure of the field $\eq(\La_T)$.
\end{prop}

{\color{black}Note that in Proposition \ref{realandalgebraicclos} (b), the existence of solutions 
 of any equation $P(Z)=0$, where $P\in \mathcal N^\cc_{\varOmega}[Z]$, $\deg P>0$, follows from the monodromy theorem and the condition (iv) in the definition of $c$-filter (cf.  \cite[proof of Theorem 2.4]{Spodzieja2}).}

\subsection{Orderings in fields of real Nash functions}
\label{orderingsineqlaDeltaa}

Let us fix a $c$-filter $\varOmega$ in $\er^T$. 

Recall that by $\partial \varOmega$ we denote the set $\bigcap_{U\in\varOmega}\overline{U}$. 
Recall also that  $\varOmega$ \emph{determines an ordering} $\succ$ in $\mathcal{N}^\er_{\varOmega}$ 
 (see Proposition \ref{realandalgebraicclos}), i.e.,  a total ordering satisfying: 
$$
f\succ g\ \Rightarrow\  f+h \succ  g+h\quad\hbox{and}\quad  
f\succ  0\,\land\, g\succ 0\ \Rightarrow\  fg \succ  0
$$
such that $f\succ 0$ iff $f>0$ on some $U\in \varOmega$. If $f\succ g$ then we also write $g\prec f$.

From \cite[Theorem 3.1, Remark 3.2 and Corollary 5.4]{Sppjm} we have

\begin{twr}\label{Corollary2} 
The following conditions are equivalent:

{\rm (a)} The field $(\mathcal{N}^\er_{\varOmega},\succ)$ is Archimedean.

{\rm (b)} There exists $x_\succ\in\partial\varOmega$ {\color{black} whose coordinates are pairwise different and the set of these coordinates 
 is algebraically independent over $\eq$}.

{\rm (c)} There exists $x_\succ\in\partial\varOmega$ such that $x_\succ \in U$ for any $U\in \varOmega$.

{\rm (d)} There exists $x_\succ\in\partial\varOmega $ such that $f\succ 0$ iff $f(x_\succ)>0$, provided $f\in \mathcal{N}^\er_{\varOmega}$.
\end{twr}

{\color{black}\begin{remark}\label{orderinarchimedean}
If $(\mathcal{N}^\er_{\varOmega},\succ)$ is an Archimedean field, where $\varOmega$ is a $c$-filter in $\er^T$, then one can assume that $T\subset \er$ and it is algebraically independent over $\eq$ and ordered in such a way that for $t_1,t_2\in T$ we have $\La_{t_1}\succ\La_{t_2}$ iff $t_1>t_2$. In fact, according to Theorem \ref{Corollary2}, it suffices to take  the set of coordinates of $x_{\succ}$ as the set $T$. 
\end{remark}}

Let $\K=\er$ or $\K=\cc$. 
Let $T\subset \er$ be an infinite set algebraically independent over~$\eq$. Let ${\bf x}_0\in \er^T$ be defined by ${\bf x}_0(t)=t$ for $t\in T$. Take the $c$-filter $\varOmega^\K_{{\bf x}_0}$ centered at ${\bf x}_0$ defined in Section \ref{cfilterssection}. 
The field $\mathcal{N}^\K_{\varOmega_{{\bf x}_0}^\K}$ will be denoted by $\mathcal{N}^\K_{{\bf x}_0}$. 
Then Theorem \ref{Corollary2} gives 

\begin{cor}\label{specialordering}
The field $\mathcal{N}^\er_{{\bf x}_0}$ is an Archimedean, real closed field which is the real closure of $\eq(\La_T)$. Moreover the function $\mathcal{N}^\er_{{\bf x}_0}\ni f\mapsto f({\bf x}_0)\in \er$ is an order preserving monomorphism.
\end{cor}

It is easy to prove that $\varOmega_{{\bf x}_0}^\er=\{U\cap \er^T:U\in \varOmega_{{\bf x}_0}^\cc\}$. 
Since any analytic function $f:U\to\cc$, where $U\subset\er^T$ is an open set, has a unique holomorphic extension $\tilde f:\tilde U\to\cc$ onto some open set $\tilde U\subset \cc^T$ with $U\subset \tilde U$, by Proposition \ref{realandalgebraicclos} we immediately obtain

\begin{cor}\label{specialordering2}
The field $\mathcal{N}^\cc_{{\bf x}_0}$ is the algebraic closure of $\eq(\La_T)$ and of $\mathcal{N}^\er_{{\bf x}_0}$. Moreover, the mapping
$$
\Psi: \mathcal{N}^\cc_{{\bf x}_0}\ni f\mapsto \re f|_{\er^T}+i\im f|_{\er^T}\in \mathcal{N}^\er_{{\bf x}_0}(i)
$$
is an isomorphism of fields, where $i^2=-1$.
\end{cor}

\subsection{{\color{black}Another} $c$-filter on $\K^T$}\label{ordering in Rm}

Let $\K=\er$ or $\K=\cc$. 
Let $m$ be a fixed positive integer and $\La$ a system of $m$ variables $\La_1,\ldots,\La_m$. 

Take any {\color{black}nonzero} $P\in \eq[\La]$. Set
$$ 
\Gamma_P=\{(\la _1,\ldots,\la _m)\in \K^m :\,P(\la _1,\ldots,\la _{m-1},\la
_m+\gamma )=0\text{ for some }\gamma \in [0,\infty )\}.
$$
We define a polynomial $\omega (P)\in \eq[\La _1,\ldots,\La_{m-1}]\setminus\{0\}$ (or a number $\oop\in\eq\setminus\{0\}$ if $m=1$) by 
 $\oop =P_0$, 
 where
\begin{equation*}\label{oop}
 P=P_0\La ^d_m+ P_1\La ^{d-1}_m+\cdots+P_d,
\end{equation*} 
and $P_i\in \eq[\La _1,\ldots,\La _{m-1}]$ (or $P_i\in\eq$ if $m=1$) for $i=0,\ldots,d$, $P_0\ne 0$.

We now define sets $\op \subset \K^m $, $P\in\eq[\La]\setminus\{0\}$, by induction on~$m$:
\begin{equation*}\label{1.4} 
\op =\K \setminus \Gamma _P\subset \K  \quad \hbox{if}\quad m=1,
\end{equation*}
\begin{equation*}\label{1.5}
\op =(\K^m \setminus \Gamma _P)\cap (\opo \times \K )\subset \K^m\quad \hbox{if}\quad m>1 .  
\end{equation*} 

By the Tarski-Seidenberg Theorem (see \cite{Seidenberg,Tarski1}), the sets $\op$ for $P\in\eq[\La]\setminus\{0\}$ are $\eq$-semialgebraic. {\color{black} Indeed, this is clear for $\K=\er$. Let us explain it in the case when
 $\K = \cc$. Then $P(\la_1,\ldots,\la_m)= P(x_1+iy_1,\ldots,x_m+iy_m)=u(x_1,\ldots,x_m,y_1,\ldots,y_m)+iv(x_1,\ldots,x_m,y_1,\ldots,y_m)$, where $i^2=-1$ and $u,v\in \eq[x_1,\ldots,x_m,y_1,\ldots,y_m]$ are the real and imaginary parts of $P$ respectively. So, 
\begin{multline*}
\Gamma_P=\{(x_1,\ldots,x_m,y_1,\ldots,y_m)\in\er^{2m}:u(x_1,\ldots,x_m+\gamma,y_1,\ldots,y_m)\\
=v(x_1,\ldots,x_m+\gamma,y_1,\ldots,y_m)=0\text{ for some }\gamma \in [0,\infty )\}
\end{multline*}
is $\eq$-semialgebraic in $\er^{2m}$, and consequently $W_P$ is $\eq$-semialgebraic in $\cc^m=\er^{2m}$.}

An argument analogous to the proof of \cite[Theorem 1.1]{Spodzieja2} gives the following 

\begin{prop}\label{realNash1}  The family ${\mathcal W}=\{W_P: P\in \eq[\La],\;{\color{black}P\ne 0}\}$ is a $c$-filter and satisfies the following conditions: 

\noindent$ R_0.  \quad \op \subset \{\la \in \K^m :\,P(\la )\not= 0\}$,

\noindent$ R_1.  \quad \op \cap W_Q=W_{PQ}$,

\noindent$ R_2.  \quad 
\op \text{  is an unbounded subset of } \K^m $,
 
\noindent$ R_3.  \quad 
 \op \text{  is open, connected and simply connected}$,

\noindent$ R_4. \quad \text{for } \K=\cc,\; 
 \op \text{ is a dense subset of $\cc^m$}$. 

\noindent$ R_5.  \quad \op = \K^m \text{ for }P=\const $, $P\ne 0$.
\end{prop}

We have (cf. \cite[Lemma 4.2]{Sppjm})

\begin{lem}\label{omegap} Let $1\le i_1<\cdots<i_m\le n$, and let $P\in \eq[\La_{i_1},\ldots,\La_{i_m}]\setminus\{0\}$. Let $Q\in\eq[\La_1,\ldots,\La_n]$ be a polynomial of the form 
\begin{equation*}\label{postac}
Q(x_1,\ldots,x_n)=P(x_{i_1},\ldots,x_{i_m}),\quad (x_1,\ldots,x_n)\in\K^n.
\end{equation*} 
Then $W_P\subset \K^m$, $W_Q\subset \K^n$, and
$$
W_Q=\{(x_1,\ldots,x_n)\in\K^n:(x_{i_1},\ldots,x_{i_m})\in W_P\}.
$$
\end{lem}

Let $T$ be a nonempty linearly ordered set with ordering $\succ$. 
For any $ t_1,\ldots, t_m\in T$ with $ t_1 \prec\cdots\prec t_m$ we consider the projection map
$$
\pi_{ t_1,\ldots, t_m}:\K^T\ni x\mapsto (x( t_1),\ldots,x( t_m))\in \K^m.
$$
We define a family $\WW_{T}^\K$ of $\eq$-semialgebraic subsets $U$ of $\K^T$ by
\begin{equation*}\label{orderfamily1}
U=(\pi_{ t_1,\ldots, t_m})^{-1}(W_P)
\end{equation*}
for any $m\in\en\setminus\{0\}$, $ t_1,\ldots, t_m\in T$, $ t_1\prec\cdots\prec t_m$ and $P\in\eq[\La_{ t_1}\ldots,\La_{ t_m}]\setminus \{0\}$. 
From  Lemma \ref{omegap} and Proposition \ref{realNash1} (cf. \cite[Proposition 4.3]{Sppjm}) we obtain

\begin{prop}\label{realfield1} $\WW_{T}^\K$ is a $c$-filter. 
\end{prop}

{\color{black}  In Section \ref{cfilterssection} we observed that  there exists a $c$-filter $\varOmega^\K_{{\bf x}_0}$ in $\K^T$ 
provided $T\subset \er$. Proposition  \ref{realfield1} generalizes this to any set $T$.} 

From the definition of the sets $W_P$ in the real and complex cases and from \cite[Corollary 4.5]{Sppjm}, Proposition \ref{realNash1} and Lemma \ref{omegap} we have

\begin{cor}\label{correstriction}
$\WW^\er_T=\{U\cap \er^T:U\in \WW^\cc_{T}\}$.
\end{cor}
{\color{black}
\begin{remark}\label{lambda1lambda2}
It is easy to see that for $\K=\er$ and $P=\La_2-\La_1\in\eq[\La_1,\La_2]$ we have $W_P=\{(\la_1,\la_2)\in\er^2:\la_2 > \la_1\}$, so $\La_2\succ_{\WW} \La_1$ for the ordering $\succ_{\WW}$ in $\mathcal{N}^\er_{\WW}$ determined by the $c$-filter $\WW$ in $\er^2$ (see Proposition \ref{realNash1}). So, for any  linearly ordered set $T$ with ordering $\succ$ 
 we have $t_1\succ t_2$ iff $\La_{t_1}\succ_{\WW_T^\er} \La_{t_2}$.
\end{remark}
}

{\color{black}
\begin{remark}\label{rempartialWTC} 
  By Proposition
  \ref{realNash1} 
 for any $U\in \WW^\cc_T$ we have $\overline{U}=\cc^T$, so $\partial \WW^\cc_T=\bigcap_{U\in \WW^\cc_T}\overline{U} =\cc^T$. On the other hand,   $\partial \WW^\er_T=\emptyset$. Indeed, take any $t\in T$. Then $U_n=\{x\in \er^T:x(t)>n\}\in \WW^\er_T$ for all $n\in \en$, and so $\partial \WW^\er_T\subset \bigcap_{n\in\en}\overline{U_n}=\emptyset$. 
\end{remark}}

We will denote by $\mathcal{N}^\K_T$ the field of Nash functions $\mathcal N^\K_{\varOmega}$, where $\varOmega=\WW^\K_{T}$ is the $c$-filter defined above. A similar argument to that for Corollary \ref{specialordering2} gives 

\begin{prop}\label{nashextensionN}
The  mapping
$$
\Psi: \mathcal{N}^\cc_T\ni f\mapsto \re f|_{\er^T}+i\im f|_{\er^T}\in \mathcal{N}^\er_T(i)
$$
is an isomorphism of fields, where $f|_{\er^T}$ is the restriction $f|_{U\cap \er^T}:U\cap \er^T\to\cc$, provided $f\in\mathcal{N}^\cc_{T}(U)$, $U\in \WW^\cc_{T}$ and $i^2=-1$. Consequently,  $\mathcal{N}^\cc_T$ is an algebraic extension of $\mathcal{N}^\er_T$ of degree $2$. Moreover, the field $\mathcal{N}^\cc_T$ is the algebraic closure of $\mathcal{N}^\er_T$.
\end{prop}
 
 \begin{remark}
 By the definition of the $c$-filter $\WW_T^\cc$,  any function $f\in \mathcal{N}_T^\cc$ is holomorphic in an open connected, simply connected and dense subset of $\cc^T$.
 \end{remark}

\subsection{Extensions of $c$-filters}\label{sectio26prime}

Let $\K=\er$ or $\K=\cc$. 
Let $(T_1, T_2)$ be a pair of nonempty disjoint linearly ordered sets with orderings $\succ_1$ and $\succ_2$, respectively. Then $T_1\cup T_2$ is linearly ordered by: for any $t,t'\in T_1\cup T_2$, $t\succ_{1,2} t'$ iff either $t\in T_2$ and $t'\in T_1$, or $t,t'\in T_1$ and $t\succ_1 t'$, or $t,t'\in T_2$ and $t\succ_2 t'$.
 Then one can consider the space $\K^{T_1}\times\K^{T_2}$ as $\K^{T_1\cup T_2}$.  
 
Under the above convention, the construction of the $c$-filter $\WW^\K_T$ gives 

\begin{prop}\label{extensioncfilter}
The $c$-filter $\WW^\K_{T_1\cup T_2}$ of subsets of $\K^{T_1}\times \K^{T_2}$ contains the families $\{U\times \K^{T_2}:U\in \WW^\K_{T_1}\}$ and $\{\K^{T_1}\times U:U\in\WW^\K_{T_2}\}$. 
\end{prop}

On account   of       the above proposition, the $c$-filter $\WW^\K_{T_1\cup T_2}$ will be called an \emph{extension} of $\WW^\K_{T_1}$ and of $\WW^\K_{T_2}$. 

It is easy to observe that the assertion  of Proposition \ref{extensioncfilter} also holds for $c$-filters centered at points {\color{black}$\varOmega^\K_{{\bf x_1}}$, $\varOmega^\K_{{\bf x}_2}$, respectively in $\K^{T_1}$, $\K^{T_2}$, provided $T_1,T_2\subset \er$ are disjoint  and their union $T_1\cup T_2$ is algebraically independent over $\eq$. In the case of arbitrary $c$-filters, a similar construction
cannot be made,  because it leads to many filters in the Cartesian product of appropriate spaces. 
For instance $\varOmega=\{(0,\ve):\ve\in\eq_+\}$ is a $c$-filter in $\er$ but there are  infinitely many $c$-filters in $\er^2$ containing  $\{U\times \er:U\in \varOmega\}$ and $\{\er\times U:U\in\varOmega\}$.}
 
Let $(T_1, T_2)$ be a pair of nonempty disjoint linearly ordered sets.

\begin{prop}\label{nashextension}
The field $\mathcal N^\K_{T_1\cup T_2}$ is an extension of $\mathcal N^\K_{T_1}$ and  $\mathcal N^\K_{T_2}$.
\end{prop}

\begin{proof} Indeed, any function $f\in\mathcal{N}^\K_{T_1}$ has a representative $f:U\to \K$, where $U\in \WW^\K_{T_1}$, which we may consider as a function $f:U\times \K^{T_2}\to\K$. So, $f\in \mathcal N^\K_{T_1\cup T_2}$. Obviously addition and multiplication extend from $\mathcal{N}^\K_{T_1}$ to $\mathcal N^\K_{T_1\cup T_2}$. Analogously we consider the case of $f\in\mathcal{N}^\K_{T_2}$.
\end{proof}

\section{A geometric model of an arbitrary differentially closed field}\label{constralgebraically}

\subsection{Derivations on a field of Nash functions}

Let $\K=\er$ or $\K=\cc$. Consider a $c$-filter $\varOmega$ 
in $\K^T$ and the $\K$-field $\mathcal N^\K_{\varOmega}$ of Nash functions.  Take any family 
$$
g=(g_t\in \mathcal N^\K_{\varOmega}: t\in T),
$$
and let $\delta_g :\mathcal N^\K_{\varOmega}\to \mathcal N^\K_{\varOmega}$ be the mapping defined by
\begin{equation}\label{eqdeltag}
\delta_g(f)=\sum_{t\in T}g_t\frac{\partial f}{\partial \La_t}\quad\hbox{for }f\in \mathcal N^\K_{\varOmega}.
\end{equation}
The mapping $\delta_g$ is well defined, because any representative of $f\in \mathcal N^\K_{\varOmega}$ depends only on a finite number of variables, so $\frac{\partial f}{\partial \La_t}\in \mathcal{N}^\K_{\varOmega}$ and the sum in \eqref{eqdeltag} is finite. We have

\begin{prop}\label{derywacjenaNash1}
The mapping $\delta_g$ defined by \eqref{eqdeltag} is a derivation (more precisely, a $\eq$-derivation) on $\mathcal N^\K_{\varOmega}$.
Moreover, any derivation on $\mathcal N^\K_{\varOmega}$ is of the form \eqref{eqdeltag}.
\end{prop}

\begin{prop}\label{differentialdiffeo}
Assume that $(K,\delta)$ is a differential field of characteristic zero, let $\varphi:K\to \mathcal N^\K_{\varOmega}$ be a $\eq$-embedding, and let $\mathcal{K}=\varphi(K)$. Then the  mapping $\delta_\varphi:\mathcal{K}\to \mathcal{K}$ defined by 
$$
\delta_\varphi(f)=\varphi (\delta(\varphi^{-1}(f)))
$$
is a derivation on $\mathcal{K}$, and $\varphi$ is a $\eq$-differential isomorphism of the differential fields $(K,\delta)$, $(\mathcal{K},\delta_\varphi)$. 
\end{prop}

\begin{proof} Obviously $\delta_\varphi$ is a $\eq$-linear mapping, and for any $f,g\in \mathcal{K}$,
\begin{equation*}
\begin{split}
\delta_\varphi(fg)&=\varphi(\delta(\varphi^{-1}(fg)))=\varphi(\delta(\varphi^{-1}(f)\varphi^{-1}(g)))\\
&=\varphi(\delta(\varphi^{-1}(f)))g+f\varphi(\delta(\varphi^{-1}(g)))=\delta_\varphi(f)g+f\delta_\varphi(g).
\end{split}
\end{equation*}
On the other hand, for any $a\in K$, 
$\varphi(\delta(a))=\varphi(\delta(\varphi^{-1}(\varphi (a))))=
\delta_\varphi (\varphi(a))$, 
 which completes the proof.
\end{proof} 

\begin{twr}\label{charactdiffclosedfield}
Let $(K,\delta)$ be a differentially closed field of characteristic zero. Then there exists an infinite set $T$ such that $(K,\delta)$ is $\eq$-differentially isomorphic to $(\mathcal N^\cc_{\varOmega},\delta_g)$ for an arbitrary $c$-filter $\varOmega$ 
 in $\cc^T$ and some family 
\begin{equation}\label{eqformfamilyg}
g=(g_t\in \mathcal N^\cc_{\varOmega}: t\in T).
\end{equation}
\end{twr}

\begin{proof} Let $T$ be a transcendence basis of $K$ over $\eq$. By Proposition \ref{differentiallyclosedtranscendence}, $T$ is an infinite set. Since $K$ is algebraically closed, being differentially closed,  Proposition \ref{realandalgebraicclos}(b) implies that $K$ is $\eq$-isomorphic to $\mathcal N^\cc_{\varOmega}$ for an arbitrary  $c$-filter $\varOmega$ 
in $\cc^T$. Then, by Propositions \ref{derywacjenaNash1} and \ref{differentialdiffeo} we see that $(K,\delta)$ is $\eq$-differentially isomorphic to ($\mathcal N^\cc_{\varOmega},\delta_g)$ for some family $g$ of the form \eqref{eqformfamilyg}.
\end{proof}

\subsection{A derivation which makes the field of Nash functions differentially closed}\label{constrdiimakesdifferentially}

Let $T$ be a linearly ordered infinite set with ordering $\succ$. Let $\varOmega$ be a $c$-filter in 
 $\cc^T$ (e.g., the one defined in Section \ref{ordering in Rm}). Set 
$$
\frakN=\mathcal N^\cc_{\varOmega}.
$$ 
Consider the ring of polynomials
$$
\frakN[Y]=\frakN [Y_{j}:j\in\en ].
$$
For any $\frakP\in\frakN[Y]$ we set 
$$
D(\frakP)=\left\{t\in T:
\frac{\partial \frakP}{\partial \La_{t}}=0\right\}.
$$
Obviously $T\setminus D(\frakP)$ is a finite set.

For $\frakP\in\frakN[Y]$ with $\deg \frakP> 0$, we set 
$$
\alpha(\frakP)=\max\left\{j\in\en: \deg_{Y_j}\frakP>0\right\},
$$
where $\deg_{Y_j}\frakP$ denotes the degree of $\frakP$ as a polynomial in $Y_j$. 
Additionally we set $\alpha(\frakP)=-1$ if $\frakP\in \frakN\setminus\{0\}$, and $\alpha(0)=-\infty$.

Define
$$
\mathcal{A}=\{(\frakP,\frakG)\in \frakN[Y]^2:
\alpha(\frakP)>\alpha(\frakG)\ge -1\}.
$$

\begin{fact}\label{factdiff2}
The sets $T$, $\frakN$, $\frakN[Y]$ and $\mathcal{A}$ have the same cardinality.
\end{fact}

{\color{black}\begin{proof} 
Since $T$ is  infinite, it has the same cardinality as the set $\operatorname{Fin}(T)$ of all finite subsets of $T$.  So, $\eq[\La_T]$ has cardinality $\operatorname{card}T$, because it is the union of the countable sets $\eq[\La_{t_1},\ldots,\La_{t_m}]$ for $\{t_1,\ldots,t_m\}\in \operatorname{Fin}(T)$.  Consequently, $\operatorname{card}\eq[\La_T,Z]=\operatorname{card}T$. Hence $\operatorname{card}\frakN=\operatorname{card}T$, because any polynomial $P\in\eq[\La_T,Z]$ determines a finite subset of $\frakN$ and the set $\eq[\La_T]$ of cardinality $\operatorname{card}T$ is contained in $\frakN$. 
Analogously, $\frakN[Y]$ is the union of the sets $\frakN[Y_{0},\ldots,Y_{m}]$, $m\in\en$, 
 so $\operatorname{card}\frakN[Y]=\operatorname{card}T$. Since $[Y_1(\frakN[Y]\setminus\{0\})]\times \{1\}\subset
 \mathcal{A}\subset \frakN[Y]^2$, we obtain $\operatorname{card}\mathcal{A}=\operatorname{card}T$.
\end{proof}}

\begin{fact}\label{factdiff3}
There exists a family of  pairwise disjoint infinite and countable subsets $T_{\frakP,\frakG}\subset T$, $(\frakP,\frakG)\in\mathcal{A}$, such that
$$
T=\bigcup_{(\frakP,\frakG)\in \mathcal{A}}T_{\frakP,\frakG}.
$$
\end{fact}

\begin{proof}
Since $T$ is infinite,  there exists a bijection $\tau:\en\times T\to T$. By Fact \ref{factdiff2} there exists a bijection $\eta:\mathcal{A}\to T$. Thus for $T_{\frakP,\frakG}=\tau(\en\times\{\eta(\frakP,\frakG)\})\subset T$, $(\frakP,\frakG)\in\mathcal{A}$, we obtain the assertion.
\end{proof}

\begin{fact}\label{factdiff4}
Let $(\frakP,\frakG)\in\mathcal{A}$. For any $t_{\frakP,\frakG,0},\ldots,t_{\frakP,\frakG,\alpha(\frakP)-1}\in D(\frakP)\cap D(\frakG)\cap {\color{black} T_{\frakP,\frakG}}$ such that $t_{\frakP,\frakG,0}\prec \cdots\prec t_{\frakP,\frakG,\alpha(\frakP)-1}$ we have  
\begin{equation*}\label{eqfactdiff41}
\frakG(\La_{t_{\frakP,\frakG,0}},\ldots,\La_{t_{\frakP,\frakG,\alpha(\frakP)-1}})\ne 0
\end{equation*}
and 
\begin{equation*}\label{eqfactdiff42}
\deg_{Y_{\alpha(\frakP)}}\frakP(\La_{t_{\frakP,\frakG,0}},\ldots,\La_{t_{\frakP,\frakG,\alpha(\frakP)-1}},Y_{\alpha(\frakP)})>0,
\end{equation*}
under the natural convention when $\alpha(\frakP)=0$. 
Moreover, points $t_{\frakP,\frakG,0},\ldots,$\linebreak $t_{\frakP,\frakG,\alpha(P)-1}\in D(\frakP)\cap D(\frakG)\cap {\color{black} T_{\frakP,\frakG}}$ such that $t_{\frakP,\frakG,0}\prec \cdots\prec t_{\frakP,\frakG,\alpha(\frakP)-1}$ always exist, provided $\alpha(\frakP)>0$.
\end{fact}

\begin{proof} If $\alpha(\frakP)=0$ then the assertion is trivial. Assume that $\alpha(\frakP)>0$. 
By the definition of   $\mathcal{A}$, the polynomial $\frakG$ depends on at most $\alpha(\frakP)-1$ first variables $Y_j$. Since $D(\frakP)\cap D(\frakG)\cap {\color{black} T_{\frakP,\frakG}}$ is an infinite set,  there exist $t_0,\ldots,t_{\alpha(\frakP)-1}\in D(\frakP)\cap D (\frakG)\cap {\color{black} T_{\frakP,\frakG}}$ such that  $t_0\prec \cdots\prec t_{\alpha(\frakP)-1}$, and hence we immediately deduce the assertion.
\end{proof}

{\color{black}Assume that we have fixed a choice of $t_{\frakP,\frakG,i}$ for $(\frakP,\frakG)\in\mathcal{A}$ as in Fact \ref{factdiff4}.}

Let $(\frakP,\frakG)\in\mathcal{A}$ and let $g_{\frakP,\frakG}\in \frakN $ be a solution of the equation
$$
\frakP(\La_{t_{\frakP,\frakG,0}},\ldots,\La_{t_{\frakP,\frakG,\alpha(\frakP)-1}},Y_{\alpha(\frakP)})=0
$$
with respect to $Y_{\alpha(\frakP)}$. Recall that $\frakN=\mathcal N^\cc_{\varOmega}$ is an algebraically closed field, so $g_{\frakP,\frakG}\in \frakN $ always exists.

Define a family $g$ of points $g_t\in\frakN$, 
 $t\in T$, by
\begin{equation}\label{selection}
g_t=\begin{cases}\La_{t_{\frakP,\frakG,i+1}}&\hbox{for }t=t_{\frakP,\frakG,i},\;i=0,\ldots,\alpha(\frakP)-2,\\
g_{\frakP,\frakG}&\hbox{for } t=t_{\frakP,\frakG,\alpha(\frakP)-1},\\
f_t &\hbox{for }t\in T_{\frakP,\frakG}\setminus\{t_{\frakP,\frakG,0},\ldots,t_{\frakP,\frakG,\alpha(\frakP)-1}\},
\end{cases}
\end{equation}
under the natural convention when $\alpha(\frakP)\in\{0,1\}$, 
where $f_t\in \frakN$ 
are arbitrary for $t\in T_{\frakP,\frakG}\setminus\{t_{\frakP,\frakG,0},\ldots,t_{\frakP,\frakG,\alpha(\frakP)-1}\}$,  
for each $(\frakP,\frakG)\in \mathcal{A}$. Since $t_{\frakP,\frakG,i}\in T_{\frakP,\frakG}$, by Fact \ref{factdiff3} the family $g$ is well defined. 
Consider the derivation
\begin{equation*}\label{defivationderivative}
\delta_g(f) =\sum_{t\in T}g_t\frac{\partial f}{\partial \La_t}\quad\hbox{for }f\in \frakN .
\end{equation*}

\begin{twr}\label{maindifferclosed1}
$(\frakN, 
\delta_g)$ is a differentially closed field.
\end{twr}

\begin{proof}
Obviously $(\frakN,\delta_g)$ is a differential field. It suffices to prove that for each pair $p, q \in \frakN\{y\}$ of differential polynomials such that 
$\ord q< \ord p$, $q\ne 0$, there is some $f\in  \frakN$ with $p(f) = 0$ and $q(f)\ne 0$. Since the field $\frakN$ is algebraically closed,  this condition obviously holds in the case $\ord p=0$. Assume that $\ord p>0$.

Since $\frakN\{y\}=\frakN[y_{\delta_g^n}:n\in\en]$, there exists a one-to-one correspondence between $\frakN\{y\}$ and $\frakN[Y]$ determined by $Y_j\mapsto y_{\delta_g^j}$ for $j\in\en$. So, for any $p,q\in \frakN\{y\}$ with ${\color{black}n=}\ord p>\ord q$, $q\ne 0$, there exist $\frakP,\frakG\in\frakN[Y]$ with {\color{black} $\frakP=p^*$ and $\frakG=q^*$, i.e.,}
$$
p=\frakP(y_0,y_{\delta},\ldots,y_{\delta^n})\quad\hbox{and}\quad q=\frakG(y_0,y_{\delta},\ldots,y_{\delta^n}),
$$
{\color{black}$\alpha(\frakP)=n\ge 0$ and $\alpha(\frakG)=\ord q$. 
Then by the definition of $\delta_g$ for $f=\La_{t_{\frakP,\frakG,0}}\in\frakN$ we have 
 $\delta^m_g(f)=\La_{t_{\frakP,\frakG,m}}$ for $0\le m\le n-1$ and $\delta^n_g(f)=g_{\frakP,\frakG}$. So, by Facts \ref{factdiff3} and \ref{factdiff4}, 
 $p(f)=0$ and $q(f)\ne 0$, which completes the proof.}
\end{proof}

From the choice of $g$ in \eqref{selection} and Theorem \ref{maindifferclosed1} we have

\begin{cor}\label{corcardialitydifferrentailclos}
The set of all derivations $\delta$ on $ \mathcal{N}^\cc_{\varOmega}$  such that  $(\mathcal{N}^\cc_{\varOmega},\delta)$ is a differentially closed field has  cardinality $2^{\operatorname{card}(T)}$.
\end{cor}

\subsection{A universal extension of a differential field}\label{diffclosedextensionsec} 
Let $T\ne\emptyset$ be a linearly ordered set. 
 Take the $c$-filter $\WW^\cc_{T}$ in  $\cc^T$ defined in Section \ref{ordering in Rm}. Consider a pair $(T_1,T_2)$ of nonempty sets, where $T_1=T$ and $T_2$  is a linearly ordered infinite set such that $T_1\cap T_2=\emptyset$ and {$\operatorname{card}T_2=\max\{\operatorname{card}T,\operatorname{card}\en\}$}. Let $\succ$ be the ordering of $T_1\cup T_2$ defined in Section \ref{sectio26prime}. 
Let 
$$
\frakN=\mathcal{N}^\cc_{T_1\cup T_2}
$$
be the extension of $\mathcal{N}^\cc_{T}$ (see Proposition \ref{nashextension}). 

Let $T_2=T_3\cup T_4$, where $T_3\cap T_4=\emptyset$ and $\operatorname{card}T_3=\operatorname{card}T_2$.

We will use similar notation to that of Section \ref{constrdiimakesdifferentially}. 
 Consider the ring of polynomials 
$\frakN[Y]=\frakN[Y_{j}:j\in\en ]$ and set 
$$
D_{T_3}(\frakP)=\left\{t\in T_3:
\frac{\partial \frakP}{\partial \La_{t}}=0\right\}\quad\hbox{for }\frakP\in\frakN[Y].
$$
Then $T_3\setminus D_{T_3}(\frakP)$ is a finite set.

Consider the set 
$
\mathcal{A}=\{(\frakP,\frakG)\in \frakN[Y]^2:
\alpha(\frakP)>\alpha(\frakG)\ge -1\}.
$

We easily see that the sets $T_2$, $\frakN$, $\frakN[Y]$  and 
$\mathcal{A}$ have the same cardinality (cf. Fact \ref{factdiff2}).
By Fact \ref{factdiff3} there exists a family of  pairwise disjoint countable subsets $T_{\frakP,\frakG}\subset T_3$, $(\frakP,\frakG)\in\mathcal{A}$, such that
$$
T_3=\bigcup_{(\frakP,\frakG)\in \mathcal{A}}T_{\frakP,\frakG}.
$$

For any $(\frakP,\frakG)\in\mathcal{A}$ the set $D_{T_3}(\frakP)\cap D_{T_3}(\frakG)\cap T_{\frakP,\frakG}$ is infinite. So there are $t_{\frakP,\frakG,\ell,0},\ldots,t_{\frakP,\frakG,\ell,\alpha(\frakP)-1}\in D_{T_3}(\frakP)\cap D_{T_3}(\frakG)\cap T_{\frakP,\frakG}$, where $1\le \ell\le k$, $k=\deg_{Y_{\alpha(\frakP)}}\frakP\ge 1$, such that 
\begin{equation}\label{eqordert1}
t_{\frakP,\frakG,1,0}\prec\cdots\prec t_{\frakP,\frakG,1,\alpha(\frakP)-1}\prec\cdots\prec t_{\frakP,\frakG,k,0}\prec\cdots\prec t_{\frakP,\frakG,k,\alpha(\frakP)-1}\; ,
\end{equation}
\begin{equation*}\label{eqfactdiff41x}
\frakG(\La_{t_{\frakP,\frakG,\ell,0}},\ldots,\La_{t_{\frakP,\frakG,\ell,\alpha(\frakP)-1}})\ne 0,\quad 1\le \ell\le k,
\end{equation*}
and 
\begin{equation}\label{eqordert2}
k=\deg_{Y_{\alpha(\frakP)}}\frakP(\La_{t_{\frakP,\frakG,\ell,0}},\ldots,\La_{t_{\frakP,\frakG,\ell,\alpha(\frakP)-1}},Y_{\alpha(\frakP)}),\quad 1\le \ell\le k.
\end{equation}
Hence there exist $h_{\frakP,\frakG,\ell}\in\mathcal{N}^\cc_{T\cup \{t_{\frakP,\frakG,\ell,0},\ldots ,t_{\frakP,\frakG,\ell,\alpha(\frakP)-1}\}}\subset \frakN$, $1\le \ell\le k $, 
 such that  
\begin{equation}\label{eqordert3}
\frakP(\La_{t_{\frakP,\frakG,\ell,0}},\ldots,\La_{t_{\frakP,\frakG,\ell,\alpha(\frakP)-1}},h_{\frakP,\frakG,\ell})=0,\quad 1\le \ell\le k,
\end{equation}
and for some $1\le j\le k$ (equivalently, for each $1\le j\le k$),
\begin{equation}\label{eqallhPQell}
\begin{split}
&h_{\frakP,\frakG,\ell}(\La_{t_{\frakP,\frakG,j,0}},\ldots,\La_{t_{\frakP,\frakG,j,\alpha(\frakP)-1}}),\; 
1\le \ell\le k,\hbox{ are all solutions counted}\\
&\hbox{with multiplicity of the equation }\frakP(\La_{t_{\frakP,\frakG,j,0}},\ldots,\La_{t_{\frakP,\frakG,j,\alpha(\frakP)-1}},Y_{\alpha(\frakP)})=0.
\end{split}
\end{equation}
Obviously $h_{\frakP,\frakG,\ell}(\La_{t_{\frakP,\frakG,j,0}},\ldots,\La_{t_{\frakP,\frakG,j,\alpha(\frakP)-1}})\in\mathcal{N}^\cc_{T\cup \{t_{\frakP,\frakG,j,0},\ldots ,t_{\frakP,\frakG,j,\alpha(\frakP)-1}\}}\subset \frakN  $ for any $1\le j\le k$. 

Define a family $h$ of points $h_t\in \frakN $, $t\in T_3$, by
\begin{equation*}\label{selectionx}
h_t=\begin{cases}\La_{t_{\frakP,\frakG,\ell,i+1}}&\hbox{for }t=t_{\frakP,\frakG,\ell,i},\; 0\le i \le \alpha(\frakP)-2,\; 1\le \ell\le k ,\\
h_{\frakP,\frakG,\ell}&\hbox{for } t=t_{\frakP,\frakG,\ell,\alpha(\frakP)-1},\;1\le \ell\le k ,\\
f_t &\hbox{for }t\in T_{\frakP,\frakG}\setminus\{t_{\frakP,\frakG,1,0},\ldots,t_{\frakP,\frakG,k,\alpha(\frakP)-1}\}
\end{cases}
\end{equation*}
under the natural convention
when $\alpha(\frakP)\in\{0,1\}$,
we take $k=\deg_{Y_{\alpha(\frakP)}}\frakP$ and $f_t\in \frakN $ are arbitrary for $t\in T_{\frakP,\frakG}\setminus\{t_{\frakP,\frakG,1,0},\ldots,t_{\frakP,\frakG,k,\alpha(\frakP)-1}\}$, for each $(\frakP,\frakG)\in \mathcal{A}$. 

Let $\delta_g$ be a derivation on $\mathcal{N}^\cc_{T}$ of the form 
\begin{equation*}
\delta_g(f) =\sum_{t\in T}g_t\frac{\partial f}{\partial \La_t}\quad\hbox{for }f\in \mathcal{N}^\cc_{T_1}
\end{equation*}
for some family $g=(g_t\in \mathcal{N}^\cc_{T}:t\in T)$. Take any family $w=(w_t\in\frakN:t\in T_4)$. Then the mapping 
$\delta:\frakN\to \frakN$ defined by
\begin{equation*}\label{eqdeltastar}
\delta (f)=\sum_{t\in T }g_t\frac{\partial f}{\partial \La_t}+\sum_{t\in T_3}h_t\frac{\partial f}{\partial \La_t}+\sum_{t\in T_4}w_t\frac{\partial f}{\partial \La_t}
\end{equation*}
is a derivation extending $\delta_g$. 
So, by an analogous argument to that for Theorem \ref{maindifferclosed1} we deduce that 

\begin{cor}\label{diffclosedextension}  $(\frakN,\delta )$ is a differentially closed field differentially extending $(\mathcal{N}^\cc_{T},\delta_g)$.
\end{cor}

\begin{proof}
Obviously $(\frakN,\delta)$ is a differential extension of $(\mathcal{N}^\cc_{T},\delta_g)$. Take any $p,q\in \frakN\{y\}$ such that $\ord p>\ord q$ and $q\ne 0$. Then $p^*=\frakP$ and $q^*=\frakG$ for some $(\frakP,\frakG)\in\mathcal{A}$. So, for $f=\La_{t_{\frakP,\frakG,1,0}}$ we have $p(f)=0$ and $q(f)\ne 0$. 
\end{proof}

Assume that $\operatorname{card}T_4=\operatorname{card}\en$. For simplicity one can assume that $T_4=\en$. 
Take a bijection $\sigma :T_4\times \en\to T_4$ and a family $v=(v_t: t\in T_4)$ defined by
$$
v_t=\La_{\sigma(s,k+1)}\quad \hbox{if } t=\sigma(s,k)\in T_4,\; (s,k)\in T_4\times \en.
$$
Let $T_{4,s}=\sigma(\{s\}\times \en)$, $s\in T_4$. Then the sets $T_{4,s}$ are countable, pairwise disjoint, and 
$$
T_4=\bigcup_{s\in T_4}T_{4,s}.
$$ 
Further, the mapping $\delta^{*}:\frakN \to \frakN$ defined by
\begin{equation*}\label{eqsemiuniversal}
\delta^{*}(f)=\sum_{t\in T }g_t\frac{\partial f}{\partial \La_t}+\sum_{t\in T_3}h_t\frac{\partial f}{\partial \La_t}+\sum_{t\in T_4}v_t\frac{\partial f}{\partial \La_t}
\end{equation*}
is a derivation extending $\delta_g$, and by  Corollary \ref{diffclosedextension}, $(\frakN,\delta^{*})$ is a differentially closed field. Moreover, we have

\begin{lem}\label{lem1}
{\rm(a)} For any $s\in T_4$ the mapping $\delta^*$ is a derivation in $\mathcal{N}^\cc_{T\cup T_{4,s}} $.

{\rm(b)} If $(\frakP,\frakG)\in \mathcal{A}$ and $\frakP\in \mathcal{N}^\cc_T[Y]$ then for any $1\le \ell\le \deg_{Y_{\alpha(\frakP)}}\frakP$ the mapping $\delta^*$ is a derivation in $\mathcal{N}^\cc_{T\cup T'} $, where 
$T'=\{t_{\frakP,\frakG,\ell,0},\ldots,t_{\frakP,\frakG,\ell,\alpha(\frakP)-1}\}$.
\end{lem}

\begin{lem}\label{corsemiuniversal} 
Let $(\mathcal{N}^\cc_{T},\delta)$ be a differential extension of a differential field $(\mathcal{L},\delta)$ 
 such that $\mathcal{N}^\cc_T$ is an  algebraic extension of $\mathcal{L}$. 
Let $\mathcal{F}\subset \mathcal{N}^\cc_{T\cup T^\circ}$, with $(T\cup T_2)\cap T^\circ=\emptyset$, be a field such that $(\mathcal{F},\delta_1)$ is a simply generated differential extension of $(\mathcal{L},\delta)$ and let  $c$ be a generator of the extension. 
Assume that $c$ is transcendental over $\mathcal{L}$.

{\rm (a)} If the sequence $(\delta_1^n(c):n\in\en)$  
is algebraically independent\footnote{i.e., $\delta_1^0(c),\ldots,\delta_1^m(c)$ are algebraically independent over $\mathcal{L}$, for any $m\in\en$.} over $\mathcal{L}$,  then for any $s\in T_4$ the mapping $\varPhi:\mathcal{F}\to\mathcal{N}^\cc_{T\cup T_{4,s}}\subset\frakN$ defined by
\begin{equation}\label{eqdefinitionofembedding1}
\varPhi(f)=f\quad\hbox{for }f\in \mathcal{L}\qquad\hbox{and}\qquad \varPhi (\delta^n_1(c))=\La_{\sigma(s,n)}\quad\hbox{for }n\in\en,
\end{equation}
is a differential embedding over $\mathcal{L}$. Moreover, the field  $\mathcal{N}^\cc_{T\cup T_{4,s}}$ is an algebraic extension of $\varPhi(\mathcal{F})$.


{\rm(b)} If 
$\delta_1^0(c),\ldots ,\delta_1^{m-1}(c)$ is the longest sequence algebraically independent over $\mathcal{L}$, 
 then there exists $(\frakP,\frakG)\in\mathcal{A}$ such that $\alpha(\frakP)=m$ 
  and for some open and connected subset $V\subset \cc^{T\cup T^\circ}$,
\begin{align}
&\frakP(\delta_1^0(c),\ldots,\delta_1^{m-1}(c), \delta_1^m(c))=0\quad \hbox{in } V, \label{eqalgzaldelta}\\
&\delta_1^m(c)=h_{\frakP,\frakG,\ell}(\delta _1^0(c),\ldots,\delta_1^{m-1}(c))\quad\hbox{in $V$ for some }1\le \ell\le \deg_{Y_m}\frakP.\label{eqdeltamform}
\end{align}
Moreover,  for any such $(\frakP,\frakG)$,  $V\subset \cc^{T\cup T^\circ}$ and  $1\le \ell\le \deg_{Y_{m}}\frakP$, the mapping $\varPhi_1:\mathcal{F}\to \mathcal{N}^\cc_{T\cup T'}\subset \frakN$ defined by 
\begin{multline}\label{eqdefinitionofembedding2}
\varPhi_1(f)=f\quad\hbox{for }\;f\in\mathcal{L},\quad\varPhi_1(\delta_1^n(c))=\La_{t_{\frakP,\frakG,\ell,n}}\;\hbox{ for }\;0\le n\le m-1,\\
\hbox{and }\; \varPhi_1 (\delta_1^n(c))=(\delta^*)^{n-m}(h_{\frakP,\frakG,\ell})\in\mathcal{N}^\cc_{T\cup T'}\;\hbox{ for }\;n\ge m,
\end{multline}
is a differential embedding over  $\mathcal{L}$, where $T'=\{t_{\frakP,\frakG,\ell,0},\ldots,t_{\frakP,\frakG,\ell,m-1}\}$. In particular the field $\mathcal{N}^\cc_{T\cup T'}$ is an algebraic extension of $\varPhi_1(\mathcal{F})$.
\end{lem}

\begin{proof} 
Let the sequence $(\delta_1^n(c):n\in\en)$ be algebraically independent over $\mathcal{L}$, let $s\in T_4$ and let $\varPhi_1:\mathcal{F}\to\frakN$ be the mapping defined by \eqref{eqdefinitionofembedding1}.  By Lemma \ref{lem1}(a), $\delta^*$ is a derivation in $\mathcal{N}^\cc_{T\cup T_{4,s}}$ and obviously $\varPhi(\mathcal{F})\subset \mathcal{N}^\cc_{T\cup T_{4,s}}$.  Moreover,
$$
\varPhi(\delta_1^n(c))=(\delta^*)^n(\La_{\sigma(s,0)})=(\delta^*)^n(\varPhi(c))\quad\hbox{for }n\in\en,
$$ 
so, $\varPhi$ is  a differential embedding over  $\mathcal{L}$. 
Obviously  $\mathcal{N}^\cc_{T\cup T_{4,s}}$ is an algebraic extension of $\varPhi(\mathcal{F})$, which gives (a).

Assume now that 
$\delta_1^0(c),\ldots ,\delta_1^{m-1}(c)$ is the longest  sequence algebraically independent over $\mathcal{L}$. Then 
 $\delta_1^n(c)\in\mathcal{N}^\cc(W)$, $0\le n \le m$, for some $W\in \WW^\cc_{T\cup T^\circ}$. Let 
\begin{align}
&F:W\ni (\la,x)\mapsto (\la,\delta_1^0(c)(\la,x),\ldots,\delta_1^{m-1}(c)(\la,x))\in\cc^T\times \cc^m,\nonumber
\end{align}
where $\la\in\cc^T$ and $x\in C^{T^\circ}$. 

We claim that there exists an open connected subset $V\subset W$ such that $F(V)\subset \cc^T\times\cc^m$  has nonempty connected interior. Indeed, since $\delta_1^n(c)$, $0\le n\le m-1$, depends on a finite number of variables, it suffices to consider the case when $T\cup T^\circ$ is a finite set. Let $X\subset \cc^T\times \cc^{T^\circ}\times\cc^m$ be the graph of $F$,  let $Y$ be the Zariski closure of $X$, and let $\pi:Y\ni(\la,x,y)\mapsto (\la,y)\in \cc^T\times \cc^m$. Take the ideal $\mathcal{I}\subset \cc[(\La_t;t\in T),(x_{t^\circ}:t^\circ\in T^\circ),(y_0,\ldots,y_{m-1})]$ of polynomials vanishing on $Y$. Since $X$ is the graph of a mapping which components are $\eq$-Nash functions on open connected set $W$, it is a connected complex analytic manifold and so, irreducible analytic subset of $W\times \cc^m$ (see \cite[Corollary 3, p.216]{Lojasiewicz}). So,  by \cite[Proposition 4, p. 217 and Corollary after Proposition 2, p. 408]{Lojasiewicz} the set $Y$ is irreducible and consequently the ideal $\mathcal{I}$ is prime. Moreover, using Gr\"obner bases (see for instance \cite[Theorem 2.4]{Robiano}, \cite[Section 9]{Gianni}, see also \cite{Basu}, \cite{GP}), we obtain that $\mathcal{I}$ is generated by polynomials with rational coefficients, and that the ideal $\mathcal{J}\subset \cc[(\La_t;t\in T),(y_0,\ldots,y_{m-1})]$ of the set $\overline{\pi(Y)}\subset \cc^T\times \cc^m$ is also generated by polynomials with rational coefficients.

By the assumption $\delta_1^0(c),\ldots ,\delta_1^{m-1}(c)$ are algebraically independent over $\mathcal{L}$ and $\mathcal{L}\subset \mathcal{N}^\cc_T$ is an algebraic extension, so they 
 are algebraically independent over $\mathcal{N}^\cc_T$ and in particular -- over $\eq(\La_T)$. So $\La_t$, $t\in T$, and $\delta_1^0(c),\ldots ,\delta_1^{m-1}(c)$ are algebraically independent over $\eq$. Since the ideal $\mathcal{J}$ is generated by polynomials with rational coefficients, \cite[Theorem 1.22]{Basu} gives that $\mathcal{J}=\{0\}$ and the set $\pi(Y)$ is a constructible (i.e., it is in the Boolean algebra generated by the  closed algebraic sets, see \cite{Mumford}) and dense subset of $\cc^T\times \cc^m $. 
 Since $Y$ is an irreducible algebraic set, there exists a proper algebraic subset $Y_0\varsubsetneq Y$ such that $\pi|_{Y\setminus Y_0}:Y\setminus Y_0\to \cc^T\times \cc^m$ is an open mapping (see \cite[Corollary 3.15]{Mumford} and the Riemann Open Mapping Theorem \cite[Theorem V.6.2]{Lojasiewicz}). Consequently $F(W)=\pi(X)$ has nonempty interior because $X\subset Y$ and $X$ contains a nonempty open subset of $Y\setminus Y_0$. This easily gives the announced claim. 

By the assumption, $\delta_1^0(c),\ldots ,\delta_1^{m}(c)$ are algebraically dependent over $\mathcal{L}$. So, there exists 
$(\frakP,\frakG)\in\mathcal{A}$ with $\alpha(\frakP)=m$ and  $\frakP\in\mathcal{L}[Y]\subset  \mathcal{N}^\cc_{T}[Y]$ such that 
$
\frakP(\delta_1^0(c),\ldots,\delta_1^{m-1}(c), \delta_1^m(c))=0.$ 
Take any such $(\frakP,\frakG)$, $t_{\frakP,\frakG,\ell,0},\ldots,t_{\frakP,\frakG,\ell,m-1}\in D_{T_3}(\frakP)\cap D_{T_3}(\frakG)\cap T_{\frakP,\frakG}$, and  $h_{\frakP,\frakG,\ell}\in\frakN$, where $1\le \ell\le k$, $k=\deg_{Y_{m}}\frakP$, for which \eqref{eqordert1} -- 
\eqref{eqallhPQell} hold. Since $\frakP\in \mathcal{N}^\cc_{T}[Y]$, we have $h_{\frakP,\frakG,\ell}\in  \mathcal{N}^\cc_{T\cup T'}$, where $T'=\{t_{\frakP,\frakG,\ell,0},\ldots,t_{\frakP,\frakG,\ell,m-1}\}$. So, $\frakP\in \mathcal{N}^\cc(U)[Y]$ and $h_{\frakP,\frakG,\ell}\in\mathcal{N}^\cc(U)$, $1\le \ell\le k$, for some $U\in \WW^\cc_{T\cup T'}$. By Lemma \ref{lem1}(b) the mapping $\delta^*$ is a derivation in $ \mathcal{N}^\cc_{T\cup T'} $.  

From Proposition \ref{realNash1}, $U$ is an open and dense subset of $\cc^T\times \cc^m$, so by the above claim for some nonempty open connected set $V\subset W$ we have that $F(V)\subset U$ is an open and connected set. Consequently, \eqref{eqalgzaldelta} holds in $V$ and by 
 \eqref{eqallhPQell}, there exists $1\le \ell_0\le k$ such that 
$$
\delta_1^m(c)(\la,x)=h_{\frakP,\frakG,\ell_0}(\la,\delta_1^0(c)(\la,x),\ldots,\delta_1^{m-1}(c)(\la,x))\quad\hbox{for } (\la,x)\in V
$$
and \eqref{eqdeltamform} holds. So, it is easy to observe that the mapping $\varPhi_1:\mathcal{F}\to \frakN$ defined by \eqref{eqdefinitionofembedding2} is an embedding over  $\mathcal{L}$.  By the definition of $\delta^*$ we conclude that $\varPhi_1$ is a differential embedding of $(\mathcal{F},\delta_1)$ in $\mathcal{N}^\cc_{T\cup T'}$ over  $\mathcal{L}$. Furthermore, the homomorphism $\varPhi_1$ transforms the transcendence basis 
$\{\delta_1^0(c),\ldots ,\delta_1^{m-1}(c)\}$ of $\mathcal{F}$ over $\mathcal{L}$ onto the transcendence basis $\{\La_t:t\in T'\}$ of $\mathcal{N}^\cc_{T\cup T'}$ over $\mathcal{N}^\cc_T$. Since $\mathcal{N}^\cc_{T\cup T'}$ is an algebraically closed field and $\mathcal{N}^\cc_T$ is an algebraic extension of $\mathcal{L}$, the field $\mathcal{N}^\cc_{T\cup T'}$ is an algebraic extension of $\varPhi_1(\mathcal{F})$, which gives (b) and completes the proof.
\end{proof}

\begin{twr}\label{coruniversal} 
$(\frakN,\delta^{*})$ is a universal extension of $(\mathcal{N}^\cc_{T},\delta_g)$.
\end{twr}

\begin{proof} We claim that $(\frakN,\delta^{*})$ is a semiuniversal extension of $(\mathcal{N}^\cc_{T},\delta_g)$. Take any finitely generated differential extension $(\mathcal{F},\delta_1)$ of $(\mathcal{N}^\cc_{T},\delta_g)$  and let  $\{c_1,\ldots,c_N\}$ be the set of generators of the extension. 
Obviously $(\mathcal{F},\delta_1)$ is equal to the quotient field of the differential domain $\mathcal{N}^\cc_{T}[\delta_1^n(c_j):n\in\en,\;1\le j\le N]$ with derivation $\delta_1$.

Let $\mathcal{F}_\nu$ be the quotient field of the domain $\mathcal{N}^\cc_{T}[\delta_1^n(c_j):n\in\en,\;1\le j\le \nu]$, $0\le \nu\le N$. Then $\mathcal{F}_0=\mathcal{N}^\cc_{T}$ and $(\mathcal{F}_{\nu+1},\delta_1)$ is a simply generated extension of $(\mathcal{F}_\nu,\delta_1)$ for $1\le \nu\le N-1$. By  \cite[Proposition II.2.3]{Kolchin} one can assume that $c_{\nu+1}$ is transcendental over $\mathcal{F}_\nu$. If the sequence $(\delta_1^n(c_1):n\in\en)$ is algebraically independent over $\mathbb{N}^\cc_T$ then by Lemma \ref{corsemiuniversal}(a) for any $s\in T_4$ the mapping $\varPhi:\mathcal{F}_1\to\mathcal{N}^\cc_{T\cup T_{4,s}}\subset \frakN$ defined by \eqref{eqdefinitionofembedding1} with $c=c_1$ is a differential embedding over $\mathcal{N}^\cc_T$ and $\mathcal{N}^\cc_{T\cup T_{4,s}}$ is an algebraic extension of $\varPhi(\mathcal{F}_1)$. Therefore, using several times Lemma \ref{corsemiuniversal}(a), we may assume that for any $1\le j\le N$ there exists $m_j\in\en$ such that $\delta_1^0(c_j),\ldots,\delta_1^{m_j-1}(c_j)$ is the longest  sequence algebraically independent over $\mathcal{F}_{j-1}$. By Proposition  \ref{differentialdiffeo} we may assume that $\mathcal{F}_j\subset\mathcal{N}^\cc_{T\cup T^\circ_j}\subset \mathcal{N}^\cc_{T\cup T^\circ}$ for some finite sets $
T^\circ_1\varsubsetneq \ldots\varsubsetneq T^\circ_N=T^\circ$ with  $(T\cup T_2)\cap T^\circ=\emptyset$, such that $\mathcal{N}^\cc_{T\cup T^\circ_j}$ is an algebraic extension of $\mathcal{F}_j$. Then $\delta_1^n(c_j)\in\mathcal{N}^\cc(W)$, $0\le n \le m_j$, for some $W\in \WW^\cc_{T\cup T^\circ}$.

Since $c_1$ is transcendental over $\mathcal{N}^\cc_T$, by Lemma \ref{corsemiuniversal}(b) there exist $(\frakP,\frakG)\in\mathcal{A}$ with $\alpha(\frakP)=m_1$, there exists $1\le \ell\le \deg_{Y_{m_1}}\!\frakP$, $T_1'=\{t_{\frakP,\frakG,\ell,0},\ldots,t_{\frakP,\frakG,\ell,m_1-1}\}$ and $h_{\frakP,\frakG,\ell}\in \mathcal{N}^\cc_{T\cup T'_1}$ such that $\varPhi_1:\mathcal{F}_1\to \frakN$ defined by \eqref{eqdefinitionofembedding2} with $c=c_1$ is a differential embedding over $\mathcal{N}^\cc_T$ and $\mathcal{N}^\cc_{T\cup T'_1}$ is an algebraic extension of $\varPhi(\mathcal{F}_1)$. So, we may assume that $\mathcal{F}_1\subset \mathcal{N}^\cc_{T\cup T'_1}$ is an algebraic extension. Then $c_2$ is transcendental over $\mathcal{F}_1$ and we may repeat the above argument with $c_2$ and the extension $\mathcal{F}_1\subset \mathcal{F}_2$. By applying Lemma \ref{corsemiuniversal}(b) $N$ times we find that $(\mathcal{F},\delta_1)=(\mathcal{F}_N,\delta_1)$ differentially embeds over $\mathcal{N}^\cc_{T}$ in $(\mathcal{N}^\cc_{T\cup T'},\delta^*)$ for some finite set $T'\subset T_3$. 
These iterations of Lemma \ref{corsemiuniversal}  are possible, because for a fixed $\frakP_0\in\frakN[Y]$, $\alpha(\frakP_0)\ge 0$, there are infinitely many $\frakG\in\frakN[Y]$ such that $(\frakP_0, \frakG) \in \mathcal{A}$, and so the family of sets $\{T_ {\frakP_0,\frakG}: (\frakP_0, \frakG) \in \mathcal{A}\}$  is infinite and for any $\frakP$, $\frakG$ we have defined all roots $h_{\frakP,\frakG,\ell}$  of $\frakP(\La_{t_{\frakP,\frakG,1,0}},\ldots,\La_{t_{\frakP,\frakG,1,\alpha(\frakP)-1}},Y_{\alpha(\frakP)})=0$, so we can choose appropriate $\ell$ for which \eqref{eqdeltamform} holds.  Summing up, $(\frakN,\delta^{*})$ is a semiuniversal extension of $(\mathcal{N}^\cc_{T},\delta_g)$.

To complete the proof it suffices to prove that for any finitely generated differential extension $(\mathcal{F},\delta^*)$ of $(\mathcal{N}^\cc_{T},\delta_g)$ in  $(\frakN,\delta^{*})$, the field  $(\frakN,\delta^{*})$ is a semiuniversal extension of $(\mathcal{F},\delta^*)$. Indeed, by the above, there are a finite set $T'\subset T_3$ and a finite union $T''=T_{4,s'_1}\cup \cdots\cup T_{4,s'_\mu}\subset T_4$ such that $\mathcal{N}^\cc_{T\cup T'\cup T''}$ is an algebraic extension of  $\mathcal{F}$. 
Let $(\mathcal{G},\delta_1)$ be a finitely generated differential extension of $(\mathcal{F},\delta^*)$. 
Then, an analogous argument as in the first two paragraphs of the proof, but using $(\mathcal{G},\delta_1)$ in place of $(\mathcal{F},\delta^*)$ and  $(\mathcal{F},\delta^*)$ in place of $(\mathcal{N}^\cc_{T},\delta_g)$, gives that  $(\mathcal{G},\delta_1)$ can be differentially embedded in $(\frakN,\delta^{*})$ over  $(\mathcal{F},\delta^*)$. This completes the proof. %
\end{proof}

\begin{remark}\label{remuniversal}
If  $\mathcal{L}=\eq$  then $(\frakN,\delta^{*})$ is the universal differential field.
\end{remark}

\section{An Archimedean ordered differentially closed field}\label{constrorderedalgebraically}

\subsection{A geometric model of an arbitrary ordered differential field}\label{arbitraryordereddifffield}

In \cite{Sppjm} we proved that there exists a one-to-one  correspondence between the family of orderings in $\eq(\La_T)$ and the family of plain filters (see \cite[Theorem 5.2, Proposition 2.4 and Corollary 2.5]{Sppjm}, cf. \cite{Brocker}). By a \emph{plain filter} we mean a $c$-filter $\varOmega$ of subsets of $\er^T$ defined by:

\smallskip
1) any $U\in{\varOmega}$ is a connected component of the complement of a proper $\eq$-algebraic set $V\subset \er^T$,

2) for any proper $\eq$-algebraic set $V\subset \er^T$, some connected component $U$ of the complement of $V$ belongs to $\varOmega$.

\smallskip

The above mentioned correspondence  is as follows:

\begin{fact}\label{correspondenceorderingplainfilter}
For any ordering $\succ$ of $\eq(\La_T)$ there exists a unique plain filter $\varOmega$ such that $f\succ 0$ iff $f>0$ on some $U\in\varOmega$. 
Conversely, any plain filter $\varOmega$ determines a  unique ordering $\succ$ of $\eq(\La_T)$ in the above way. 
\end{fact}

Since any ordering in $\mathcal{N}^\er_{\varOmega}$ is uniquely determined by an ordering in $\eq(\La_T)$, from the above fact we obtain (cf. Theorem \ref{charactdiffclosedfield} for differentially closed fields)

\begin{cor}\label{charactdiffclosedfieldreal}
Let $(K,\delta)$ be an ordered differentially closed field. Then there exists an infinite set $T$ such that $(K,\delta)$ is $\eq$-differentially order isomorphic to $(\mathcal N^\er_{\varOmega},\delta_g)$ for some $c$-filter $\varOmega$ 
 in $\er^T$ and some family 
\begin{equation}\label{eqformfamilygreal}
g=(g_t\in \mathcal N^\er_{\varOmega}: t\in T).
\end{equation}
\end{cor}

\begin{proof} Let $T$ be the transcendence basis of $K$ over $\eq$. By Corollary \ref{orderedtranscendence}, $T$ is an infinite set. Since $K$ is a real closed field, being ordered and differentially closed,  Proposition \ref{realandalgebraicclos}(a) shows that $K$ is $\eq$-order isomorphic to $\mathcal N^\er_{\varOmega}$ for some plain filter $\varOmega$ of subsets of $\er^T$. Then, by Propositions \ref{derywacjenaNash1} and \ref{differentialdiffeo} we see that $(K,\delta)$ is $\eq$-differentially order isomorphic to ($\mathcal N^\cc_{\varOmega},\delta_g)$ for some family $g$ of the form \eqref{eqformfamilygreal}.
\end{proof}

\subsection{A derivation which makes an Archimedean Nash field ordered  differentially closed}\label{Archimedeandifferentiallyclosed}

Let $T\subset \er$ be an infinite set algebraically independent over~$\eq$ ordered by the usual ordering $>$ on $\er$. Let $\varOmega=\varOmega^\er_{{\bf x}_0}$ be the $c$-filter of subsets of $\er^T$ centered at ${\bf x}_0\in\er^T$, defined by \eqref{defarchordcenbtered} in Section \ref{orderingsineqlaDeltaa}. 
 Set 
 $$
 \frakN=\mathcal N^\er_{{\bf x}_0}.
 $$
By Theorem \ref{Corollary2}, the field $\frakN$ is Archimedean, where the ordering $\succ$ in $\frakN$ is described by $f \succ 0$ iff $f({\bf x}_0)>0$. Set
$$
\frakN_{{\bf x}_0}=\{f({\bf x}_0): f\in \frakN\}.
$$

\begin{remark}\label{remarkseries}
By the definition of the $c$-filter $\varOmega^\er_{{\bf x}_0}$, each $f\in \frakN$ is a real analytic function in a neighbourhood of ${\bf x}_0$, or more precisely, $f$ is a germ of real analytic function at ${\bf x}_0$. Consequently, one can consider the elements $f$ as sums of power series centered at ${\bf x}_0$ in a finite number of variables.
\end{remark}

By Corollary \ref{specialordering} we have

\begin{fact}\label{isomorphismKN}
$\frakN_{{\bf x}_0}$ is a real closed field order  isomorphic to $\frakN$.
\end{fact} 

We will adopt the notation of Section \ref{constrdiimakesdifferentially}. Consider the ring of polynomials
$$
\frakN[Y]=\frakN 
[Y_{j}:j\in\en ].
$$
For a polynomial $\frakR\in \frakN[Y]$ of the form 
$$
\frakR(Y_0,\ldots,Y_k)=\sum_{j_0,\ldots,j_k\ge 0}f_{j_0,\ldots,j_k}Y_0^{j_0}\cdots Y_k^{j_k}
$$
where $f_{i_0,\ldots,i_k}\in\frakN$ for all $i_0,\ldots i_k$, we denote by $\frakR_{{\bf x}_0}$  the polynomial in $\frakN_{{\bf x}_0}[Y]$ defined by
$$
\frakR_{{\bf x}_0}(Y_0,\ldots,Y_k)=\sum_{j_0,\ldots,j_k\ge 0}f_{j_0,\ldots,j_k}({\bf x}_0)Y_0^{j_0}\cdots Y_k^{j_k}.
$$

Consider the  sets 
$$
\mathcal{B}_{k,n}=\{(\frakP,\frakG_1,\ldots,\frakG_n)\in \frakN[Y]^{n+1}:
k=\alpha(\frakP)\ge \alpha(\frakG_s)\ge -1,\;s=1,\ldots,n\}
$$
for $k,n\in\en$, and let
\begin{multline}\label{eqq1qn}
\mathcal{Z}=\bigcup_{k,n=1}^\infty \bigg\{(\frakP,\frakG_1,\ldots,\frakG_n,f_0,\ldots, f_k)\in \mathcal{B}_{k,n}\times \frakN^{k+1}:\frakP(f_0,\ldots,f_k)=0,\\ 
\frac{\partial \frakP}{\partial x_k}(f_0,\ldots,f_k)\ne 0,\; 
\frakG_s(f_0,\ldots,f_k)\succ 0,\;s=1,\ldots,n\bigg\}.
\end{multline}
We immediately obtain the following fact (cf. Fact \ref{factdiff2}).

\begin{fact}\label{factdiff2arch}
The sets $T$, $\frakN$, $\frakN[Y]$ and $\mathcal{Z}$ have the same cardinality.
\end{fact}

\begin{fact}\label{factdiff3arch}
There exists a family of  pairwise disjoint infinite and countable subsets $T_{\frakP,\frakG_1,\ldots,\frakG_n,f_0,\ldots, f_k}\subset T$, $(\frakP,\frakG_1,\ldots,\frakG_n,f_0,\ldots, f_k)\in\mathcal{Z}$, such that
$$
T=\bigcup_{(\frakP,\frakG_1,\ldots,\frakG_n,f_0,\ldots, f_k)\in \mathcal{Z}}T_{\frakP,\frakG_1,\ldots,\frakG_n,f_0,\ldots, f_k}.
$$
\end{fact}

\begin{proof}
Since  $T$ is infinite, there exists a bijection $\tau:\en\times T\to T$. By Fact \ref{factdiff2arch} there exists a bijection $\eta:\mathcal{Z}\to T$. Thus setting
$$
T_{\frakP,\frakG_1,\ldots,\frakG_n,f_0,\ldots, f_k}=\tau(\en\times\{\eta(\frakP,\frakG_1,\ldots,\frakG_n,f_0,\ldots, f_k)\})\subset T
$$
for $(\frakP,\frakG_1,\ldots,\frakG_n,f_0,\ldots, f_k)\in\mathcal{Z}$, we obtain the assertion.
\end{proof}

\begin{prop}\label{factdiff4real}
Let $z=(\frakP,\frakG_1,\ldots,\frakG_n,f_0,\ldots, f_k)\in\mathcal{Z}$, $k=\alpha(\frakP)$. For any 
\begin{equation}\label{pointalwaysexist}
t_{z,0},\ldots,t_{z,k-1}\in D(\frakP)\cap D(\frakG_1)\cap\ldots\cap D(\frakG_n)\cap T_z
\end{equation}
such that $t_{z,0} <\cdots < t_{z,k-1}$ there are $r_{z,0},\ldots,r_{z,k-1}\in \eq\setminus\{0\}$, such that  
\begin{equation}\label{eqfactdiff42real}
\frakP(r_{z,0}\La_{t_{z,0}},\ldots,r_{z,k-1}\La_{t_{z,k-1}},f_{z})=0
\end{equation}
and
\begin{equation}\label{eqfactdiff41real}
\frakG_s(r_{z,0}\La_{t_{z,0}},\ldots,r_{z,k-1}\La_{t_{z,k-1}},f_z)\succ 0,\quad s=1,\ldots,n,
\end{equation}
for some $f_z\in \frakN$. 
Moreover, points in \eqref{pointalwaysexist} 
 such that $t_{z,0} <\cdots < t_{z,k-1}$ always exist.
\end{prop}

\begin{proof} By definition of   $\mathcal{Z}$, the polynomial $\frakG_s$ depends on {\color{black}at most the first $k+1$ variables} $Y_j$. Since $T_{z}$ is infinite, there exist $t_{z,0},\ldots,t_{z,k}\in D(\frakP)\cap D(\frakG_1)\cap\cdots\cap D(\frakG_n)\cap T_z$ such that  $t_{z,0} < \cdots < t_{z,k}$. The set of coordinates of ${\bf x}_0$ is algebraically independent over $\eq$, so ${\bf x}_{0}(t_{z,j})\ne 0$ for $j=0,\ldots,k$. 

Let $\xi_j=f_j({\bf x}_0)\in\frakN_{{\bf x}_0}$, $j=0,\ldots, k$. In view of the choice of the point $z$, 
\begin{equation*}
\begin{split}
&\frakP_{{\bf x}_0}(\xi_0,\ldots,\xi_k)=0, \quad\frac{\partial \frakP_{{\bf x}_0}}{\partial x_k}(\xi_0,\ldots,\xi_k)\ne 0,\\ 
&(\frakG_s)_{{\bf x}_0}(\xi_0,\ldots,\xi_k) > 0,\quad s=1,\ldots,n.
\end{split}
\end{equation*}
For any $r_j\in\frakN_{{\bf x}_0}\setminus \{0\}$ sufficiently close to 
$\frac{\xi_j}{{\bf x_{0}}(t_{z,j})}$ for $j=0,\ldots,k$ 
we have 
\begin{equation}\label{eqinequalityholds}
(\frakG_s)_{{\bf x}_0}(r_0 {\bf x}_{0}(t_{z,0}),\ldots,r_k {\bf x}_{0}(t_{z,k})) > 0,\quad s=1,\ldots,n,
\end{equation}
and moreover $r_j {\bf x}_{0}(t_{z,j})=r_jF_j({\bf x}_0) \in\frakN_{{\bf x}_0}$, where $F_j(\La_T)=\La_{t_{z,j}}$ for $j=0,\ldots,k$. Then there exists $\varepsilon>0$ such that any point  of the  set
{\color{black}$$
\mathcal{U}_\varepsilon=\bigg\{r=(r_0,\ldots,r_k)\in \frakN_{{\bf x}_0}^{k+1} :\left|r_j- \frac{\xi_j}{{\bf x}_{0}(t_{z,j})}\right|<\varepsilon,\;j=0,\ldots,k\bigg\}
$$ 
satisfies \eqref{eqinequalityholds}.  Since $\frac{\partial \frakP_{{\bf x}_0}}{\partial x_k}(\xi_0,\ldots,\xi_k)\ne 0$ and $\frakN_{{\bf x}_0}$ is real closed,  the function
$$
\frakN_{{\bf x}_0}\ni \zeta \mapsto \frakP_{{\bf x}_0}(\xi_0,\ldots,\xi_{k-1},\zeta)\in\frakN_{{\bf x}_0}
$$
changes  sign at $\xi_k$. Thus there are $a,b\in \frakN_{{\bf x}_0}$ such that $a<b$ and
\begin{equation}\label{eqineqab}
\left|a-\xi_k\right|<\varepsilon|{\bf x}_0(t_{z,k}) |,\quad \left|b-\xi_k\right|<\varepsilon|{\bf x}_0(t_{z,k}) |
\end{equation}
and
$$
\frakP_{{\bf x}_0}(\xi_0,\ldots,\xi_{k-1},a)\frakP_{{\bf x}_0}(\xi_0,\ldots,\xi_{k-1},b)< 0.
$$
Since $\eq$ is a dense subset of $\frakN_{{\bf x}_0}$,  there exists $r=(r_0,\ldots,r_k)\in \mathcal{U}_\varepsilon\cap \eq^{k+1}$ such that $r_j\ne 0$ for $j=0,\ldots,k$, and
$$
\frakP_{{\bf x}_0}(r_0 {\bf x}_{0}(t_{z,0}),\ldots,r_{k-1} {\bf x}_{0}(t_{z,k-1}),a)\frakP_{{\bf x}_0}(r_0 {\bf x}_{0}(t_{z,0}),\ldots,r_{k-1} {\bf x}_{0}(t_{z,k-1}),b) < 0.
$$
As $\frakN_{{\bf x}_0}$ is real closed, this implies that there exists $\xi^*\in \frakN_{{\bf x}_0}$ such that $ a<\xi^*<b$ and
\begin{equation}\label{eqPboldx}
\frakP_{{\bf x}_0}(r_0 {\bf x_{0}}(t_{z,0}),\ldots,r_{k-1} {\bf x_{0}}(t_{z,k-1}),\xi^*)=0,
\end{equation}
and by \eqref{eqineqab},
$$
\left|\frac{\xi^*}{{\bf x}_0(t_{z,k})}-\frac{\xi_k}{{\bf x}_0(t_{z,k})}\right|<\varepsilon.
$$
Hence, $\left(r_0,\ldots,r_{k-1},\frac{\xi^*}{{\bf x}_0(t_{z,k})}\right)\in \mathcal{U}_\varepsilon$, and consequently  
\begin{equation}\label{eqequqlityandineq}
(\frakG_s)_{{\bf x}_0}(r_0 {\bf x}_{0}(t_{z,0}),\ldots,r_{k-1} {\bf x}_{0}(t_{z,k-1}),\xi^*) > 0,\quad s=1,\ldots,n.
\end{equation}

By definition of $\frakN_{{\bf x}_0}$ there exists $f_z\in \frakN$ such that $f_z({\bf x}_0)=\xi^*$. Moreover, $r_{z,j}:=r_j\in\eq$ and so $r_{z,j}\La_{t_{z,j}}\in \frakN$ for $j=0,\ldots,k-1$. Now, {\color{black}\eqref{eqPboldx}, \eqref{eqequqlityandineq} and Fact \ref{isomorphismKN} immediately give} the assertion.}
\end{proof}

{\color{black}Assume that for any $z=(\frakP,\frakG_1,\ldots,\frakG_n,f_0,\ldots, f_k)\in\mathcal{Z}$,  we have chosen  points
$$
t_{z,0},\ldots,t_{z,k-1}\in D(\frakP)\cap D(\frakG_1)\cap\cdots\cap D(\frakG_n)\cap {T}_z,
$$ 
where $k=\alpha(\frakP)$, such that $t_{z,0} < \cdots < t_{z,k-1}$ and $r_{z,0},\ldots,r_{z,k-1}\in \eq \setminus\{0\}$, and $f_z\in \frakN$ as in Proposition \ref{factdiff4real}, i.e.,  \eqref{eqfactdiff42real} and \eqref{eqfactdiff41real} hold.} 

Define a family $g$ of points $g_t\in\frakN$, $t\in T$, by
\begin{equation}\label{selectionreal}
g_t=\begin{cases}\frac{r_{z,i+1}}{r_{z,i}}\La_{t_{z,i+1}}&\hbox{for }t=t_{z,i},\;i=0,\ldots,\alpha(\frakP)-2,\\
\frac{1}{r_{z,\alpha(\frakP)-1}}f_z&\hbox{for } t=t_{z,\alpha(\frakP)-1},\\
h_t &\hbox{for }t\in T_{z}\setminus \{ t_{z,0},\ldots,t_{z,\alpha(\frakP)-1}\},
\end{cases}
\end{equation}
where $h_t\in \frakN$ are arbitrary for $t\in T_{z}\setminus \{ t_{z,0},\ldots,t_{z,\alpha(\frakP)-1}\}$,  
for each $z=(\frakP,\frakG_1,\ldots,$ $\frakG_n,f_0,\ldots, f_{\alpha(\frakP)})\in\mathcal{Z}$.

Consider the following derivation on $\frakN$:
\begin{equation}\label{defivationderivativereal}
\delta_g(f) =\sum_{t\in T}g_t\frac{\partial f}{\partial \La_t}\quad\hbox{for }f\in \frakN.
\end{equation}

\begin{twr}\label{maindifferclosed1real}
$(\frakN,\delta_g)$ is an ordered differentially closed field.
\end{twr}

\begin{proof}
Obviously $(\frakN,\delta_g)$ is an ordered differential field and by Corollary \ref{specialordering}, $\frakN$ is real closed. Take any $p,q_1,\ldots,q_n\in\frakN\{y\}$ such that $k=\ord p\ge \ord q_j$, $1\le j\le n$, and any $f_0,\ldots,f_k\in\frakN$ such that $p^*(f_0,\ldots,f_k)=0$, $\frac{\partial p^*}{\partial x_k}(f_0,\ldots,f_k)\ne 0$ and $q_j^*(f_0,\ldots,f_k)\succ 0$, $1\le j\le n$. Then  $z=(p^*,q_1^*,\ldots,q_n^*,f_0,\ldots, f_k)\in\mathcal{Z}$ and $k=\alpha(p^*)$. Since $r_{z,j}\in\eq$, by \eqref{selectionreal} for $f=r_{z,0}\La_{t_{z,0}}$ we have 
$$
{\color{black}\delta(f)=r_{z,0}\delta(\Lambda_{t_{z,0}})=r_{z,1}\Lambda_{z,1},\ \ldots,\  \delta^{k-1}(f)=r_{z,k-1}\Lambda_{z,k-1},\;\;\delta^k(f)=f_z.}
$$
So, Proposition \ref{factdiff4real}, similarly to the proof of Theorem \ref{maindifferclosed1}, shows that $p(f)=0$ and $q_j(f)\succ 0$, $1\le j\le n$, which gives the assertion. 
\end{proof}

From the choice of $g$ in \eqref{selectionreal} and Theorem \ref{maindifferclosed1real} we have

\begin{cor}\label{corcardialitydifferrentailclos}
The set of all derivations $\delta$ on  $\mathcal{N}^\er_{{\bf x}_0}$  such that  $(\mathcal{N}^\er_{{\bf x}_0},\delta)$ is an ordered  differentially closed field has cardinality $2^{\operatorname{card}(\mathcal{N}^\er_{{\bf x}_0})}$.
\end{cor}

{\color{black}\begin{remark}
By Corollary \ref{corequivrealdiffclosed}, to construct a derivation $\delta$ on $\frakN$ such that $(\frakN,\delta)$ becomes ordered differentially closed, it suffices to consider the  set
\begin{multline*}
\mathcal{Z}=\bigcup_{k,n=1}^\infty \bigg\{(\frakP,\ve,f_0,\ldots,f_k)\in \frakN[Y]\times \eq_+\times \frakN^{k+1}:k=\alpha(\frakP)\ge 0,\\[-4pt]
\frakP(f_0,\ldots,f_k)=0,\;
\frac{\partial \frakP}{\partial x_k}(f_0,\ldots,f_k)\ne 0\bigg\}
\end{multline*}
instead of the one defined in \eqref{eqq1qn}, and repeat the construction in Proposition \ref{factdiff4real} without taking into consideration the polynomials $\frakG_1, \ldots, \frakG_n$.
\end{remark}}


\begin{remark}\label{proprealdiffi}
Let  $(\mathcal{N}^\er_{{\bf x}_0},\delta_g)$ be the ordered differentially closed field with derivation $\delta_g$ defined by \eqref{selectionreal} and \eqref{defivationderivativereal}. By Proposition \ref{closedordereddifferential} (see also  \cite{Singer2}), the field $\mathcal{N}^\er_{{\bf x}_0}(i)$ with the  derivation 
$$
\delta(f_1+if_2)=\delta_g(f_1)+i\delta_g(f_2),
$$
extending  $\delta_g$, is a differentially closed field.

Indeed, since $i$ is   algebraic over $\mathcal{N}^\er_{{\bf x}_0}$, it follows that $\delta$ is the unique derivation in $\mathcal{N}^\er_{{\bf x}_0}(i)$ 
extending $\delta_g$. Thus Proposition \ref{closedordereddifferential} gives the assertion.
\end{remark}

\begin{remark}
By Remarks \ref{remarkseries}, \ref{proprealdiffi} and Corollary \ref{specialordering2} we see that any function $f\in \mathcal{N}^\cc_{{\bf x}_0}=\mathcal{N}^\er_{{\bf x}_0}(i)$ is  holomorphic in a neighborhood of ${\bf x}_0$ in $\cc^T$. Consequently, one can consider the elements $f$ as sums of power series centered at ${\bf x}_0$ in a finite number of complex variables (or as germs of holomorphic functions).
\end{remark}

\vskip.5cm
 \small
\noindent{\bf Acknowledgements.}  I would like to cordially
thank  Piotr J\k{e}drzejewicz,  Krzysztof Kurdyka and Andrzej Nowicki for their valuable remarks and advice during the preparation of this paper. I  also  thank the referee for many valuable comments that helped to avoid errors and many shortcomings in the paper.

 \end{document}